\documentclass[number]{elsarticle}

\usepackage[utf8]{inputenc}
\usepackage[TS1,T1]{fontenc}
\usepackage{xspace}

\usepackage{amssymb, amsmath}
\usepackage{amsthm}
\usepackage[all]{xy}

\usepackage[english]{babel}
\usepackage{hyperref}
\usepackage{lmodern}

\newcommand{\omi}[1]{\buildrel { \buildrel{#1}\over{\vee} } \over .}
\newcommand\bbbone{{ \mathchoice {1\mskip-4mu\mathrm{l} } {1\mskip-4mu\mathrm{l} }{1\mskip-4.5mu\mathrm{l} } {1\mskip-5mu\mathrm{l}} }}

\newcommand\dd{\text{\textup{d}}} 
\newcommand\ds{\text{\textup{s}}} 
\newcommand\hd{\widehat{\dd}} 
\newcommand\hR{\widehat{R}}
\newcommand{\grast}{\bullet} 
\newcommand\cdotaction{\mathord{\cdot}} 
\newcommand\exter{{\textstyle\bigwedge}} 
\newcommand\ad{{\text{\textup{ad}}}} 
\newcommand\sign{{\text{\textup{sign}}}}
\newcommand\hor{{\text{\textup{Hor}}}} 
\newcommand\inv{{\text{\textup{Inv}}}} 
\newcommand\bas{{\text{\textup{Basic}}}} 
\newcommand\der{{\text{\textup{Der}}}} 
\newcommand\Int{{\text{\textup{Int}}}} 
\newcommand\ensvide{{\varnothing}} 

\newcommand\Ad{{\text{\textup{Ad}}}} 
\newcommand\lie{{\text{\textup{Lie}}}}
\newcommand\equ{{\text{\textup{equ}}}} 
\newcommand\Out{{\text{\textup{Out}}}}
\newcommand{\tla}{{\lienotation{TLA}}}
\newcommand{\raR}{{\widetilde{R}}}

\newcommand{\hnabla}{{\widehat{\nabla}}}

\newcommand{\hdNC}{\hd_{\text{NC}}}

\DeclareMathOperator{\Aut}{Aut} 
\DeclareMathOperator{\End}{End} 
\DeclareMathOperator{\tr}{Tr} 

\newcommand\varnotation[1]{{\mathcal{#1}}}
\newcommand\algnotation[1]{{\mathbf{#1}}}
\newcommand\lienotation[1]{{\mathbf{\mathsf{#1}}}}
\newcommand\grnotation[1]{{\mathsf{#1}}}
\newcommand\evnotation[1]{{\mathnormal{#1}}}
\newcommand\modnotation[1]{{\boldsymbol{#1}}}

\newcommand\varA{{\varnotation{A}}}
\newcommand\varE{{\varnotation{E}}}

\newcommand\varL{{\varnotation{L}}}
\newcommand\varM{{\varnotation{M}}}
\newcommand\varP{{\varnotation{P}}}

\newcommand\algA{{\algnotation{A}}}
\newcommand\algB{{\algnotation{B}}}
\newcommand\algzero{{\grnotation{0}}}

\newcommand\lieA{{\lienotation{A}}}
\newcommand\lieB{{\lienotation{B}}}
\newcommand\lieL{{\lienotation{L}}}

\newcommand\modM{{\modnotation{M}}}

\newcommand\evE{{\evnotation{E}}}

\newcommand\kg{{\mathfrak g}}

\newcommand\kD{{\mathfrak D}}
\newcommand\kS{{\mathfrak S}} 
\newcommand\kX{{\mathfrak X}}
\newcommand\kY{{\mathfrak Y}}
\newcommand\ksl{{\mathfrak{sl}}}

\newcommand\gN{{\mathbb N}}
\newcommand\gR{{\mathbb R}}
\newcommand\gC{{\mathbb C}}

\newcommand\sfX{{\mathsf X}}
\newcommand\sfY{{\mathsf Y}}

\newcommand\caL{{\mathcal L}}
\newcommand\caN{{\mathcal N}}
\newcommand\caZ{{\mathcal Z}}

\newtheorem{theorem}{Theorem}[section]
\newtheorem{proposition}[theorem]{Proposition}

\newtheorem{corollary}[theorem]{Corollary}

\theoremstyle{definition}
\newtheorem{example}[theorem]{Example}
\newtheorem{definition}[theorem]{Definition}

\theoremstyle{remark}

\hypersetup{
pdfauthor={Serge Lazzarini, Thierry Masson},
pdftitle={Connections on Lie algebroids and on derivation-based noncommutative geometry},
pdfsubject={Lie algebroids, noncommutative geometry, connections},
pdfcreator={pdflatex},
pdfproducer={pdflatex},
pdfkeywords={Lie algebroids, noncommutative geometry, connections, forms}
}

\hypersetup{
plainpages=false,
colorlinks=true,
linkcolor=black, 
anchorcolor=black, 
citecolor=black, 
urlcolor=black, 
menucolor=black, 
filecolor=black, 
bookmarksopen=true,
bookmarksnumbered=true}

\numberwithin{equation}{section}

\usepackage[shadow,backgroundcolor=red!20]{todonotes}

\begin{document}

\begin{frontmatter}
\title{Connections on Lie algebroids and\\ on derivation-based noncommutative geometry}

\author{Serge Lazzarini}
\ead{serge.lazzarini@cpt-univ-mrs.fr}

\author{Thierry Masson}
\ead{thierry.masson@cpt-univ-mrs.fr}

\address{Centre de Physique Théorique\tnoteref{umr}\\
Case postale 907, CNRS-Luminy\\
F--13288 Marseille Cedex 9, France}

\tnotetext[umr]{Unité Mixte de Recherche (UMR 6207) du CNRS et des Universités Aix-Marseille I, Aix-Marseille II et de l'Université du Sud Toulon-Var. Unité affiliée à la FRUMAM Fédération de Recherche 2291.}%




\begin{abstract}
In this paper we show how connections and their generalizations on transitive Lie algebroids are related to the notion of connections in the framework of the derivation-based noncommutative geometry. In order to compare the two constructions, we emphasize the algebraic approach of connections on Lie algebroids, using a suitable differential calculus. Two examples allow this comparison: on the one hand, the Atiyah Lie algebroid of a principal fiber bundle and, on the other hand, the space of derivations of the algebra of endomorphisms of a $SL(n, \gC)$-vector bundle. Gauge transformations are also considered in this comparison.
\end{abstract}

\end{frontmatter}

\newpage

\tableofcontents

\section{Introduction}

Connections on principal fiber bundles and their associated covariant derivatives on vector bundles admit several generalizations. In this paper, we shall deal with two contexts in which such generalizations have been proposed and well-studied.

In noncommutative geometry \cite{Conn94}, there is a natural notion of covariant derivative on modules which mimics the ordinary properties of a covariant derivative. Using differential calculi on noncommutative associative algebras, this notion can be expressed in terms of differential forms. Using these structures to build out gauge field theories, the noncommutative algebra is considered as a substitute for a principal fiber bundle together with its gauge group \cite{Connes1990qp,Chamseddine2006ep,Mass30,Mass32}.

Lie algebroids are now widely studied as a natural generalization of vector fields on a manifold (\cite{AlmeMoli85}, \cite{Kara86}, \cite{Wein87} and references in \cite{Mack05a} and \cite{CraiFern2006a}). Natural notions of covariant derivatives can be defined, where elements of a Lie algebroid are mapped into first order differential operators on a vector bundle. 

In this paper we show that these two different generalizations of the ordinary notion of connections turn out to be strongly related between themselves. We perform an explicit comparison which relies on Atiyah Lie algebroids to which a particular noncommutative situation is related. 

In Section~\ref{Relationtoordinarygeometryandnoncommutativegeometry} the comparison is achieved by using explicit descriptions of involved spaces of differential forms in terms of which connections are written. Gauge symmetries and their actions on connections are also compared. 

On the one hand, for Atiyah Lie algebroids, the space of forms is identified to the space of basic forms in a larger differential calculus which involves the space of de~Rham forms on the underlying principal fiber bundle. This permits to compare, in equation~\eqref{eq-basicconnectiononeform}, $1$-forms associated to connections on Atiyah Lie algebroids and to connections on the principal fiber bundle. Likely, $1$-forms associated to generalizations of connections on Atiyah Lie algebroids are described in equation~\eqref{eq-basicgeneralizedconnectiononeform} in terms of $1$-forms on the principal fiber bundle.

%

On the other hand, we consider the derivation-based noncommutative geometry of the algebra of endomorphisms of a $SL(n, \gC)$-vector bundle for which the space of derivations of this algebra is a transitive Lie algebroid \cite{Mass14}. We show in Proposition~\ref{prop-identificationLiealgebroidsDerivationsAtiyah} that it is isomorphic to the Atiyah Lie algebroid associated to the underlying $SL(n, \gC)$-principal fiber bundle and we show in Theorems~\ref{thm-derAconnectionsonEandnoncommutativeconnections} and \ref{thm-GeneralizedconnectionsandtracelessNCconnections} that the two notions of ``generalized'' connections in this situation are actually the same. 


The noncommutative geometry we are interested in uses the algebraic language of differential forms. A similar language can be used on transitive Lie algebroids where a convenient differential calculus can help to describe connections without making any reference to covariant derivatives on vector bundles. Inspired by examples in noncommutative geometry, we depart from the usual interpretation of Lie algebroids as generalization of vector fields to get closer to an interpretation of transitive Lie algebroids as generalizations of (an infinitesimal version of) principal fiber bundles, where the kernel of the transitive Lie algebroid represents the ``fiber structure''. This interpretation is consistent with the case of Atiyah Lie algebroids of principal bundles (see \cite{MR1129261} for instance).

We shall be concerned only with ``infinitesimal'' constructions in the sense that we only consider (transitive) Lie algebroids. Accordingly, we shall not deal with ``global'' constructions related to Lie groupoids, which would address, in particular, integrability questions \cite{CraiFern2006a}.

\section{General definitions and properties of Lie algebroids}
\label{generalodefinitionsandproperties}

\subsection{Transitive Lie algebroids}

In this section we fix notations about transitive Lie algebroids. We refer to the monograph \cite{Mack05a} for details. In order to get closer to noncommutative structures, we have written some of the structures presented here in an algebraic language.

\begin{definition}
\label{def-liealgebroidalgebraic}
Let $\varM$ be a smooth manifold. A Lie algebroid $\lieA$ is a finite projective module over $C^\infty(\varM)$ equipped with a Lie bracket $[-,-]$ and a $C^\infty(\varM)$-linear Lie morphism, the anchor, $\rho : \lieA \rightarrow \Gamma(T\varM)$ such that
\begin{equation*}
[\kX, f \kY] = f [\kX, \kY] + (\rho(\kX)\cdotaction f) \kY
\end{equation*}
for any $\kX, \kY \in \lieA$ and $f \in C^\infty(\varM)$ where $\Gamma(T\varM)$ denotes as usual the space of smooth vector fields on $\varM$.
\end{definition}

The notation $\lieA \overset{\rho}{\rightarrow} \Gamma(T\varM)$ will be used for a Lie algebroid $\lieA$ over the manifold $\varM$ with anchor $\rho$. Using the Serre-Swan theorem, the requirement that $\lieA$ be a finite projective module makes the definition equivalent to the usual one using fiber bundle theory.

Let $\lieA \overset{\rho_\lieA}{\rightarrow} \Gamma(T\varM)$ and $\lieB \overset{\rho_\lieB}{\rightarrow} \Gamma(T\varM)$ be two Lie algebroids over the same base manifold $\varM$. A morphism of Lie algebroids is a morphism of Lie algebras and $C^\infty(\varM)$-modules $\varphi: \lieA \rightarrow \lieB$ compatible with the anchors : $\rho_\lieB \circ \varphi = \rho_\lieA$. 

\begin{definition}
A Lie algebroid $\lieA \overset{\rho}{\rightarrow} \Gamma(T\varM)$ is transitive if $\rho$ is surjective.
\end{definition}

The kernel $\lieL = \ker \rho$ of a transitive Lie algebroid is a Lie algebroid with null anchor on $\varM$. It is well-known that there exists a locally trivial vector bundle in Lie algebras $\varL$ such that $\lieL = \Gamma(\varL)$. The Lie bracket on $\lieL$ is inherited from the Lie bracket of the Lie algebra on which the fiber bundle $\varL$ is modeled.

One can summarize the structure of a transitive Lie algebroid in the short exact sequence of Lie algebras and $C^\infty(\varM)$-modules
\begin{equation*}
\xymatrix@1{{\algzero} \ar[r] & {\lieL} \ar[r]^-{\iota} & {\lieA} \ar[r]^-{\rho} & {\Gamma(T\varM)} \ar[r] & {\algzero}}.
\end{equation*}

\subsection{Representation}
\label{subsec-representation}

Let $\varE$ be a vector bundle over the manifold $\varM$. Denote by $\algA(\varE)$ the associative algebra of smooth sections of the fiber bundle of endomorphisms $\End(\varE)$ of $\varE$ and denote by $\kD(\varE)$ the space of first-order differential operators on $\varE$ with scalar symbols. Let $\sigma : \kD(\varE) \rightarrow \Gamma(T\varM)$ be the symbol map. Then 
\begin{equation*}
\xymatrix@1{{\algzero} \ar[r] & {\algA(\varE)} \ar[r]^-{\iota} & {\kD(\varE)} \ar[r]^-{\sigma} & {\Gamma(T\varM)} \ar[r] & {\algzero}}
\end{equation*}
is the transitive Lie algebroid of derivations of $\varE$ \cite{MR585879,MR1958838}.

A representation of a transitive Lie algebroid $\lieA \xrightarrow{\rho} \Gamma(T\varM)$ on a vector bundle $\varE \rightarrow \varM$ is a morphism of Lie algebroids $\phi : \lieA \rightarrow \kD(\varE)$ \cite{Mack05a}.

In that case one has the commutative diagram of exact rows
\begin{equation}
\label{eq-diagramrepresentation}
\xymatrix
{
{\algzero} \ar[r] & {\lieL} \ar[r]^-{\iota}\ar[d]^-{\phi_\lieL}  & {\lieA} \ar[r]^-{\rho} \ar[d]^-{\phi} & {\Gamma(T\varM)} \ar[r] \ar@{=}[d]  & {\algzero}
\\
{\algzero} \ar[r] & {\algA(\varE)} \ar[r]^-{\iota} & {\kD(\varE)} \ar[r]^-{\sigma} & {\Gamma(T\varM)} \ar[r] & {\algzero}
}
\end{equation}
where $\phi_\lieL : \lieL \rightarrow \algA(\varE)$ is a $C^\infty(\varM)$-linear morphism of Lie algebras.

\section{Differential calculi and connections}
\label{ConnectionsontransitiveLiealgebroids}

\subsection{Differential structures}
\label{differentialstructures}

Given a representation, one can introduce a differential calculus in the following generic way \cite[Definition~7.1.1]{Mack05a}.

\begin{definition}
For any $p \in \gN$, let $\Omega^p(\lieA, \varE)$ be the linear space of $C^\infty(\varM)$-multilinear antisymmetric maps from $\lieA^p$ to $\Gamma(\varE)$ (smooth sections). For $p=0$ one defines $\Omega^0(\lieA, \varE) = \Gamma(\varE)$. The graded space $\Omega^\grast(\lieA, \varE) = \bigoplus_{p \geq 0} \Omega^p(\lieA, \varE)$ is equipped with the natural differential $\hd_\phi : \Omega^p(\lieA, \varE) \rightarrow \Omega^{p+1}(\lieA, \varE)$ defined on any $\omega \in \Omega^p(\lieA, \varE)$ by the Koszul formula
\begin{multline*}
(\hd_\phi \omega)(\kX_1, \dots, \kX_{p+1}) = \sum_{i=1}^{p+1} (-1)^{i+1} \phi(\kX_i)\cdotaction\omega(\kX_1, \dots \omi{i} \dots, \kX_{p+1})\\
+ \sum_{1 \leq i < j \leq p+1} (-1)^{i+j} \omega([\kX_i, \kX_j], \kX_1, \dots \omi{i} \dots \omi{j} \dots, \kX_{p+1})
\end{multline*}
where $\omi{i}$ means omission of the $i$th argument.
\end{definition}

In this definition, $\phi(\kX)\cdotaction s$ is the action of the first order differential operator $\phi(\kX)$ on a section $s \in \Gamma(\varE)$. Because $\phi$ is a morphism of Lie algebras, one has $\hd_\phi^2 = 0$.

In the following, we are interested in three particular differential calculi constructed for three specific representations of $\lieA$.

Let $\varE = \varM \times \gC$. Then $\kD(\varE) = \Gamma(T\varM)$ and the natural representation we consider is the anchor map of $\lieA$.

\begin{definition}
\label{def-formsvaluesfunctions}
$(\Omega^\grast(\lieA), \hd_\lieA)$ is the graded commutative differential algebra of forms on $\lieA$ with values in $C^\infty(\varM) = \Gamma(\varM \times \gC)$ associated to the previous representation.
\end{definition}

This space of forms is the one usually considered in the theory of Lie algebroids \cite{ArAbCrai09,Crai03a,CraiFern2006a,CraiFern2009a,Fern03a,Mack05a}. It has been also used to define cohomologies on Lie algebroids. We will not enter into this subject here, and we refer to the aforementioned references for details.

Let $\varE = \varL$ be the vector bundle for which $\lieL = \Gamma(\varL)$. There is a natural representation of $\lieA$ on $\varL$, called the adjoint representation. It is defined as follows. For any $\kX \in \lieA$ and any $\ell \in \lieL$, the Lie bracket $\ad_\kX(\ell) = [\kX, \ell]$ is defined to be the unique element in $\lieL$ such that $\iota([\kX, \ell]) = [\kX, \iota(\ell)]$, see \cite{Mack05a} for instance. 

\begin{definition}
\label{def-formsvalueskernel}
$(\Omega^\grast(\lieA, \lieL), \hd)$ is the graded differential Lie algebra of forms on $\lieA$ with values in the kernel $\lieL$ associated to the adjoint representation.
\end{definition}

Notice that this differential space is naturally a graded differential module on the graded commutative differential algebra $\Omega^\grast(\lieA)$.

Let $\varE$ be a vector bundle which supports a representation $\phi$ of $\lieA$. Then there is a natural induced representation $\widetilde{\phi}$ of $\lieA$ on the vector bundle $\End(\varE)$ defined by $\widetilde{\phi}(\kX) \cdotaction a = [\phi(\kX), a]$ for any $\kX \in \lieA$ and any $a \in \algA(\varE) = \Gamma(\End(\varE))$.

\begin{definition}
\label{def-formsvalueskernelrepresentation}
$(\Omega^\grast(\lieA, \algA(\varE)), \hd_\varE)$ is a graded differential algebra of forms on $\lieA$ with values in the associative algebra $\algA(\varE)$ for the previously defined representation.
\end{definition}

Notice that $(\Omega^\grast(\lieA, \varE), \hd_\phi)$ is a graded differential module on this graded differential algebra.

\subsection{Cartan operations}

Let $(\Omega^\grast(\lieA, \varE), \hd_\phi)$ be the graded differential space associated to a representation of $\varE$. For any $\kX \in \lieA$ and $p \geq 1$, the inner operation $i_\kX : \Omega^p(\lieA,\varE) \rightarrow \Omega^{p-1}(\lieA,\varE)$ is well defined. By convention, $i_\kX$ is zero on $\Omega^0(\lieA,\varE)$.

We associated to the inner operation its Lie derivative $L_\kX = \hd_\phi i_\kX + i_\kX \hd_\phi : \Omega^p(\lieA,\varE) \rightarrow \Omega^{p}(\lieA,\varE)$. Then, for any $\kX, \kY \in \lieA$ and $f \in C^\infty(\varM)$, one has
\begin{align}
\label{eqs-cartanoperation}
i_{f \kX} &= f i_{\kX},
&
i_\kX i_\kY + i_\kY i_\kX &= 0,
&
[L_\kX, i_\kY] &= i_{[\kX, \kY]},
&
[L_\kX, L_\kY] &= L_{[\kX, \kY]}.
\end{align}

When the differential complex is an algebra ($\Omega^\grast(\lieA)$ or $\Omega^\grast(\lieA, \algA(\varE))$ for instance), $i_\kX$ is a graded derivation of degree $-1$ and $L_\kX$ is a graded derivation of degree $0$.

These structures lead naturally to the following notion (see \cite{MR1739368} for instance).

\begin{definition}
\label{def-cartanoperationLiealgebroids}
Let $\lieB$ be a Lie algebroid over $\varM$. A Cartan operation of $\lieB$ on $(\Omega^\grast(\lieA, \varE), \hd_\phi)$ is given by the following data: for any $\kX \in \lieB$, for any $p \geq 1$, there is a map $i_\kX : \Omega^p(\lieA, \varE) \rightarrow \Omega^{p-1}(\lieA, \varE)$ such that the relations in \eqref{eqs-cartanoperation} hold for any $\kX, \kY \in \lieB$, $f \in C^\infty(\varM)$ where $L_\kX = \hd_\phi i_\kX + i_\kX \hd_\phi$.

We will denote by $(\lieB, i, L)$ such a Cartan operation on $(\Omega^\grast(\lieA, \varE), \hd_\phi)$.
\end{definition}

For instance, the kernel $\lieL$ of $\lieA$ defines a natural Cartan operation when the map $i$ is restricted to elements in $\lieL \subset \lieA$.

Given such a Cartan operation $(\lieB, i, L)$, one can define horizontal, invariant and basic elements in $\Omega^\grast(\lieA, \varE)$: $\Omega^\grast(\lieA, \varE)_\hor$ will denote the the graded subspace of horizontal elements, $\Omega^\grast(\lieA, \varE)_\inv$ the graded subspace of invariant elements and $\Omega^\grast(\lieA, \varE)_\bas = \Omega^\grast(\lieA, \varE)_\hor \cap \Omega^\grast(\lieA, \varE)_\inv$ the graded subspace of basic elements.

%
%

This notion of Cartan operation will be relevant in order to characterize differential forms on Atiyah Lie algebroids in sub-section~\ref{ex-atiyahliealgebroid}.

\subsection{Gauge transformations}

A locally trivial vector bundle $\varE$ gives rise to the natural group of vertical automorphisms, which plays an important role in the theory of connections.

\begin{definition}
\label{def-gaugegroupofarepresentation}
The gauge group of $\varE$ is the group of vertical automorphisms of the vector bundle $\varE$. It is denoted by $\Aut(\varE)$. This is the group of invertible elements in $\algA(\varE)$.
\end{definition}

If further structures are defined on $\varE$ then it is possible to restrict the gauge group to vertical automorphisms compatible with these structures. For instance, an hermitean structure on $\varE$ would give rise to the group of unitary elements in $\algA(\varE)$ for the natural involution on this algebra induced by the hermitean structure. We will not use such an extra structure here, but many of the definitions and properties given below where gauge elements appear can be adapted to such a restriction.

Gauge transformations have been used to study the moduli space of irreducible Lie algebroids connections in \cite{Kriz08a}. The infinitesimal version of the gauge group of $\varE$ is the Lie algebroid $\algA(\varE)$. The diagram \eqref{eq-diagramrepresentation} then suggests to consider infinitesimal version of gauge transformations on a transitive Lie algebroid $\lieA$ in the following way:

\begin{definition}
\label{def-infinitesimalgaugetransformations}
An infinitesimal gauge transformation on the transitive Lie algebroid $\lieA$ is an element $\xi \in \lieL$. Such an infinitesimal gauge transformation acts vertically on $\varE$ through the representation $\phi_\lieL : \lieL \rightarrow \algA(\varE)$.
\end{definition}

This definition, applied to the Atiyah Lie algebroid gives the usual infinitesimal gauge transformations of the principal fiber bundle (see~\ref{ex-atiyahliealgebroid}).

If $\lieL$ is equipped with supplementary structures, for instance a metric, infinitesimal gauge transformations can be required to be compatible in a suitable way with these structures. In that case, the space of infinitesimal gauge transformations can be strictly included in $\lieL$.

\subsection{Connections}

The ordinary notion of connection on transitive Lie algebroid is the following.

\begin{definition}[\cite{Mack05a}]
Let $\lieA \xrightarrow{\rho} \Gamma(T\varM)$ be a transitive Lie algebroid. A connection on $\lieA$ is a splitting $\nabla : \Gamma(T\varM) \rightarrow \lieA$ as $C^\infty(\varM)$-modules of the short exact sequence
\begin{equation}
\label{eq-secconnectionontransitivealgebroid}
\xymatrix@1@C=25pt{{\algzero} \ar[r] & {\lieL} \ar[r]_-{\iota} & {\lieA} \ar[r]_-{\rho} & {\Gamma(T\varM)} \ar[r] \ar@/_0.7pc/[l]_-{\nabla} & {\algzero}}
\end{equation}

The curvature of $\nabla$ is defined to be the obstruction of being a morphism of Lie algebras:
\begin{equation*}
R(X,Y) = [\nabla_X, \nabla_Y] - \nabla_{[X,Y]}
\end{equation*}
\end{definition}

For any connection $\nabla$ there is an algebraic object which allows concrete computations. It is defined as follow. For any $\kX \in \lieA$, let $X = \rho(\kX)$. Then $\kX - \nabla_X \in \ker \rho$, so that there is an element $\alpha(\kX) \in \lieL$ such that
\begin{equation}
\label{eq-decompositionwithconnection}
\kX = \nabla_X - \iota \circ \alpha(\kX)
\end{equation}
The map $\alpha : \lieA \rightarrow \lieL$ is a morphism of $C^\infty(\varM)$-modules.

\begin{proposition}
\label{connectiononeformanscurvatureinalgebraicterms}
Let $\nabla$ be a connection on the transitive Lie algebroid $\lieA$. The associated morphism $\alpha : \lieA \rightarrow \lieL$ is an element of $\Omega^1(\lieA, \lieL)$ normalized on $\iota \circ \lieL$ by the relation
\begin{equation}
\label{eq-normalizationoneformconnection}
\alpha \circ \iota(\ell) = -\ell
\end{equation}
for any $\ell \in \lieL$. It is called the connection $1$-form of the connection $\nabla$.
Any $1$-form $\alpha \in \Omega^1(\lieA, \lieL)$ normalized as in \eqref{eq-normalizationoneformconnection} defines a connection on $\lieA$.

The $2$-form 
\begin{equation}
\label{eq-curvaturetwoformglobal}
\hR = \hd \alpha + \frac{1}{2} [\alpha, \alpha]
\end{equation}
is horizontal for the Cartan operation of $\lieL$ on $(\Omega^\grast(\lieA, \lieL), \hd)$. With obvious notations, one has $\iota\circ \hR(\kX, \kY) = \iota \left( (\hd \alpha)(\kX, \kY) + [\alpha(\kX), \alpha(\kY)]\right) = R(X,Y)$.
$\hR \in \Omega^2(\lieA, \lieL)$ is called the curvature $2$-form of $\nabla$. It satisfies the Bianchi identity
\begin{equation*}
\hd \hR + [\alpha, \hR] = 0
\end{equation*}
\end{proposition}

\begin{proof}
These assertions can be checked directly.
\end{proof}

\begin{proposition}
\label{prop-covariantderivativelieAlieL}
Let $\alpha$ be the connection $1$-form of a connection $\nabla$ on $\lieA$. Then the map
\begin{align*}
D : \Omega^\grast(\lieA, \lieL) & \rightarrow \Omega^{\grast+1}(\lieA, \lieL)
&
D\eta &= \hd \eta + [\alpha, \eta]
\end{align*}
called the covariant differential of $\alpha$ on $\Omega^\grast(\lieA, \lieL)$, satisfies
\begin{equation*}
D^2 \eta = [\hR, \eta]
\end{equation*}
for any $\eta \in \Omega^\grast(\lieA, \lieL)$.
\end{proposition}

\begin{proof}
This is just a straightforward calculation using standard relations on graded differential Lie algebras.
\end{proof}

The connection $1$-form $\alpha$ is called connection reform in \cite{Mack05a}, but our sign convention \eqref{eq-normalizationoneformconnection} is just the opposite. The covariant differential defined above is the same as the one introduced in \cite[p.~192]{Mack05a}. Then, relation \eqref{eq-curvaturetwoformglobal} can be compared with the formula given in \cite[Proposition~5.2.21]{Mack05a} where the covariant differential is used instead of the differential $\hd$. The Bianchi identity can be expressed using this covariant differential as $D \hR = 0$.

Let $\varE$ be a vector bundle.

\begin{proposition}
Let $\nabla^\varE : \Gamma(T\varM) \rightarrow \kD(\varE)$ be a connection on the transitive Lie algebroid $\kD(\varE)$ and let $g \in \Aut(\varE)$ be a gauge element. Then the map $\nabla^{\varE, g} = g^{-1} \circ \nabla^\varE \circ g : \Gamma(T\varM) \rightarrow \kD(\varE)$ is a connection on $\kD(\varE)$. This defines the action of the gauge group $\Aut(\varE)$ on the space of connections on the Lie algebroid $\kD(\varE)$.
\end{proposition}

\begin{proof}
This is just a straightforward verification.
\end{proof}

We have seen that on $\lieA$ only infinitesimal gauge transformations make sense. We can define in a coherent way the action of infinitesimal gauge transformations using the previous situation as an archetype and the fact that, given a representation $\phi : \lieA \rightarrow \kD(\varE)$, a connection on $\lieA$ induces a connection on $\kD(\varE)$.

\begin{definition}
\label{def-infinitesimalgaugetransformationofaconnection}
The action of an infinitesimal gauge transformation $\xi \in \lieL$ on a connection $\nabla : \Gamma(T\varM) \rightarrow \lieA$ is given by the $C^\infty(\varM)$-linear map $[\nabla, \iota(\xi)] : \Gamma(T\varM) \rightarrow \lieA$.
\end{definition}

One has $[\nabla_X, \iota(\xi)] \in \iota \circ \lieL$ for any $X \in \Gamma(T\varM)$ and one can consider the map $X \mapsto \nabla^\xi_X = \nabla_X + [\nabla_X, \iota(\xi)] + O(\xi^2)$ as a formal development in Taylor series in $\xi$ of a connection $\nabla^\xi$ on $\lieA$. Using the language of forms, one has:
\begin{proposition}
\label{prop-infinitesimalgaugetransformationontheconnectiononeform}
The connection $1$-form associated to the (formal) connection $\nabla^\xi$ is given by $\alpha^\xi = \alpha + (\hd \xi + [\alpha, \xi]) + O(\xi^2)$ where $\alpha$ is the connection $1$-form associated to $\nabla$.

One has $L_\xi \alpha = - (\hd \xi + [\alpha, \xi])$ so that the action of an infinitesimal gauge transformation on the connection $1$-form is given by the Lie derivative of the connection $1$-form along the direction of the infinitesimal gauge transformation.

The infinitesimal action of a gauge transformation $\xi$ on the curvature is given by $(X,Y) \mapsto [R(X,Y), \xi] \in \lieL$. On the curvature $2$-form  $\hR$, the infinitesimal action of $\xi$ is given by $(\kX, \kY) \mapsto [\hR(\kX, \kY), \xi] = - (L_\xi \hR)(\kX, \kY)$.
\end{proposition}

\begin{proof}
Straightforward computations.
\end{proof}

The notion of connections on Lie algebroids admits generalizations under different names (\cite{Fern02a} and references therein): $\lieA$-connections or $\lieA$-derivatives. Let $\varE$ be a vector bundle over $\varM$.

\begin{definition}
\label{def-AconnectionsonE}
A $\lieA$-connection on $\varE$ is a linear map $\hnabla^\varE : \lieA \rightarrow \kD(\varE)$ compatible with the anchor maps such that
\begin{align*}
\hnabla^\varE_{f\kX} s &= f \hnabla^\varE_\kX s,
&
\hnabla^\varE_{\kX} (s_1 + s_2) &= \hnabla^\varE_\kX s_1 + \hnabla^\varE_\kX s_2,
&
\hnabla^\varE_{\kX} (f s) &= (\rho(\kX) \cdotaction f) s + f \hnabla^\varE_\kX s,
\end{align*}
for any $\kX \in \lieA$, for any $f \in C^\infty(\varM)$ and for any sections $s, s_1, s_2 \in \Gamma(\varE)$.

The curvature of $\hnabla^\varE$ is defined as
$\iota \circ \hR^\varE(\kX,\kY) = [\hnabla^\varE_{\kX}, \hnabla^\varE_{\kY}] - \hnabla^\varE_{[\kX, \kY]}$
where we use the fact that the expression in the right-hand side is $C^\infty(\varM)$-linear on $\Gamma(\varE)$ so that it is in $\iota(\algA(\varE))$.
\end{definition}

Such a $\lieA$-connection on $\varE$ is a ``generalized representation'' (on $\varE$) in the sense that the compatibility with Lie bracket is not required anymore and its curvature measures the failure to be a morphism of Lie algebras. A flat $\lieA$-connection is then a representation in the sense of \ref{subsec-representation} \cite{Fern02a}.

In Definition~\ref{def-AconnectionsonE}, the Lie algebroid $\lieA$ is considered as a replacement of the space of vector fields on the base manifold $\varM$, so that $\hnabla^\varE$ is a generalization of the ordinary notion of covariant derivatives on $\varE$. In particular, $\lieA$ is not required to be a transitive Lie algebroid. 

Differential structures appear in the following easy to prove proposition.

\begin{proposition}
The space of $\lieA$-connections on $\varE$ is an affine space modeled over the vector space $\Omega^1(\lieA, \algA(\varE))$.
\end{proposition}

Suppose now that $\phi : \lieA \rightarrow \kD(\varE)$ is a representation of $\lieA$ on $\varE$, and let $\hnabla^\varE : \lieA \rightarrow \kD(\varE)$ be a $\lieA$-connection on $\varE$. Then for any $\kX \in \lieA$, the difference $\hnabla^\varE_\kX - \phi(\kX)$ is in the kernel of $\sigma$, so that there exists a $1$-form $\omega^\varE \in \Omega^1(\lieA, \algA(\varE))$ such that
\begin{equation}
\label{eq-oneformofalieAconnectiononE}
\hnabla^\varE_\kX = \phi(\kX) + \iota \circ \omega^\varE(\kX)
\end{equation}
This accordingly completes the diagram~\eqref{eq-diagramrepresentation} as
\begin{equation*}
\xymatrix@R=35pt
{
{\algzero} \ar[r] & {\lieL} \ar[r]^-{\iota}\ar[d]^-{\phi_\lieL}  & {\lieA} \ar[r]^-{\rho} \ar[d]^-{\phi}_-{\hnabla^\varE} \ar[ld]_(0.35)*!/^6pt/{\omega^\varE} & {\Gamma(T\varM)} \ar[r] \ar@{=}[d]  & {\algzero}
\\
{\algzero} \ar[r] & {\algA(\varE)} \ar[r]^-{\iota} & {\kD(\varE)} \ar[r]^-{\sigma} & {\Gamma(T\varM)} \ar[r] & {\algzero}
}
\end{equation*}

\begin{proposition}
\label{prop-connectiononeformcurvaturetwoforms}
A $\lieA$-connection on $\varE$, $\hnabla^\varE : \lieA \rightarrow \kD(\varE)$, is completely determined by the $1$-form $\omega^\varE \in \Omega^1(\lieA, \algA(\varE))$ through the relation~\eqref{eq-oneformofalieAconnectiononE}. 

One has $\hR^\varE \in \Omega^2(\lieA, \algA(\varE))$, which, for any $\kX, \kY \in \lieA$, can be computed as
\begin{equation*}
\hR^\varE(\kX,\kY) =  (\hd_\varE \omega^\varE) (\kX,\kY) + [\omega^\varE(\kX), \omega^\varE(\kY)]
\end{equation*}
The $2$-form $\hR^\varE$ satisfies the Bianchi identity $\hd_\varE \hR^\varE + [\omega^\varE, \hR^\varE] = 0$.

The form $\omega^\varE$ is called the connection $1$-form of $\hnabla^\varE$ and $\hR^\varE$ is its curvature $2$-form.
\end{proposition}

\begin{proposition}
\label{prop-gaugegroupactiononlieAconnectionsonvarE}
Let $g \in \Aut(\varE)$ be a gauge group element and let $\hnabla^\varE : \lieA \rightarrow \kD(\varE)$ be a $\lieA$-connection on $\varE$. Then $\kX \mapsto \hnabla^{\varE, g}_\kX = g^{-1} \circ \hnabla^\varE_\kX \circ g$ is a $\lieA$-connection on $\varE$.
The connection $1$-form of $\hnabla^{\varE, g}$ is given by
$\omega^{\varE, g} = g^{-1} \omega^{\varE} g + g^{-1} \hd_\varE g$
and the curvature $2$-form is given by $\hR^{\varE, g} = g^{-1} \hR^\varE g$. In all these expressions products with $g$ or $g^{-1}$ take place in $\algA(\varE)$. 
\end{proposition}

Let $\nabla^\varE : \Gamma(T\varM) \rightarrow \kD(\varE)$ be an ordinary connection on the transitive Lie algebroid $\kD(\varE)$. Then $\hnabla^\varE : \lieA \rightarrow \kD(\varE)$ given by $\hnabla^\varE_\kX = \nabla^\varE_{\rho(\kX)}$ defines a $\lieA$-connection on $\varE$. It is easy to verify that this association $\nabla^\varE \mapsto \hnabla^\varE$ embeds the space of connections of the transitive Lie algebroid $\kD(\varE)$ into the space of $\lieA$-connections on $\varE$ and that this injection is compatible with the notions of curvature and gauge transformations.

As mentioned before, Lie algebroids are commonly used as natural replacement of the space of vector fields on a manifold. We would like to use transitive Lie algebroids as replacement of (an infinitesimal version of) principal fiber bundles. This is possible because any principal fiber bundle gives rise to its Atiyah transitive Lie algebroid which encodes most of the differential and geometrical structures used in relation to connection theory \cite{Mack05a,MR1129261}. In this point of view, a vector bundle which carries a representation of a transitive Lie algebroid is a replacement of an associated vector bundle.

In this spirit, we define another notion of generalized connection in a purely algebraic way.

\begin{definition}
\label{def-generalizedconnections}
A generalized connection $1$-form on the transitive Lie algebroid $\lieA$ is a $1$-form $\omega \in \Omega^1(\lieA, \lieL)$.
The curvature of a generalized connection $\omega$ is the $2$-form $\hR = \hd \omega + \frac{1}{2}[\omega, \omega] \in \Omega^2(\lieA, \lieL)$.
\end{definition}

As for infinitesimal gauge transformations, if $\lieL$ carries additional structures, some compatibility relations may be required on this $1$-form, so that \textsl{a priori} the space of ``compatible'' generalized connections is strictly smaller than $\Omega^1(\lieA, \lieL)$.

As seen in Prop.~\ref{connectiononeformanscurvatureinalgebraicterms}, the usual notion of connection on $\lieA$ gives rise to a $\lieL$-normalized $1$-form $\alpha$. This definition says that \emph{any} $1$-form can now plays the role of a connection on $\lieA$. This definition is motivated by some relations we will establish in \ref{ex-derivationsendomorphismalgebraliealgebroid} with connections in noncommutative geometry.

Given a representation $\phi : \lieA \rightarrow \kD(\varE)$ of $\lieA$, a generalized connection $1$-form $\omega \in \Omega^1(\lieA, \lieL)$ induces a $\lieA$-connection on $\varE$ by the relation 
\begin{equation*}
\hnabla^\varE_\kX = \phi(\kX) + \iota \circ \phi_\lieL \circ \omega(\kX)
\end{equation*}
and its curvature is given by $\hR^\varE = \phi_\lieL \circ \hR$. Obviously, all the $\lieA$-connections on $\varE$ can not be obtained in this way.

Using this relation between generalized connection $1$-forms on $\lieA$ and induced $\lieA$-connections on $\varE$, we can define the infinitesimal action of gauge transformations on generalized connection $1$-forms in a coherent way.

\begin{definition}
The action of the infinitesimal gauge transformation $\xi$ on the generalized connection $1$-form $\omega$ is given by the $1$-form $\omega^\xi = \omega + (\hd \xi + [\omega, \xi]) + O(\xi^2)$.
\end{definition}

In an obvious way, the space of ordinary connections on $\lieA$ is contained in the space of generalized connection $1$-forms and this inclusion is compatible with the notions of curvature and (infinitesimal) gauge transformation.

\section{Relation to ordinary geometry and to noncommutative geometry}
\label{Relationtoordinarygeometryandnoncommutativegeometry}

\subsection{Trivial Lie algebroid}
\label{trivialliealgebroid}

Let $\varM$ be a manifold and $\kg$ be a finite dimensional Lie algebra. Consider the vector bundle $\varA = T\varM \oplus (\varM \times \kg)$. The space of smooth sections $\lieA = \Gamma(\varA)$ of $\varA$ is a Lie algebroid for the following structure:
\begin{align}
\label{eq-liebrackettrivialliealgebroid}
\rho(X \oplus \gamma) &= X, 
&
[X \oplus \gamma, Y \oplus \eta] &= [X,Y] \oplus (X \cdotaction \eta - Y \cdotaction \gamma + [\gamma,\eta])
\end{align}
for any $X,Y \in \Gamma(T\varM)$ and $\gamma,\eta \in \Gamma(\varM \times \kg) \simeq C^\infty(\varM)\otimes \kg$.

$\lieA$ is called the trivial Lie algebroid over $\varM$ modeled on the Lie algebra $\kg$. We will use the (not common but convenient) notation $\tla(\varM, \kg)$ for this Lie algebroid (Trivial Lie Algebroid).

$\tla(\varM, \kg)$ is transitive and the vector bundle defining the kernel of $\rho$ is the trivial vector bundle $\varL = \varM \times \kg$. Trivial Lie algebroids are special cases of Atiyah Lie algebroids introduced in the next sub-section. The triviality comes from the triviality of the underlying principal fiber bundle.

The graded commutative differential algebra $(\Omega^\grast(\lieA), \hd_\lieA)$ identifies with the total complex of the bigraded commutative algebra $\Omega^\grast(\varM) \otimes \exter^\grast \kg^\ast$ equipped with the two differential operators
\begin{align*}
\dd : \Omega^\grast(\varM) \otimes \exter^\grast \kg^\ast &\rightarrow \Omega^{\grast+1}(\varM) \otimes \exter^\grast \kg^\ast\\
\ds : \Omega^\grast(\varM) \otimes \exter^\grast \kg^\ast &\rightarrow \Omega^\grast(\varM) \otimes \exter^{\grast+1} \kg^\ast
\end{align*}
where $\dd$ is the de~Rham differential on $\Omega^\grast(\varM)$, and $\ds$ is the Chevalley-Eilenberg differential on $\exter^\grast \kg^\ast$ so that $\hd_\lieA = \dd + \ds$.

In the same way, the graded differential Lie algebra $(\Omega^\grast(\lieA, \lieL), \hd)$ identifies with the total complex of the bigraded Lie algebra $\Omega^\grast(\varM) \otimes \exter^\grast \kg^\ast \otimes \kg$ equipped with the differential $\dd$ and the Chevalley-Eilenberg  differential $\ds'$ on $\exter^\grast \kg^\ast \otimes \kg$ for the adjoint representation of $\kg$ on itself, so that $\hd = \dd + \ds'$. We will use the compact notation $(\Omega^\grast_\tla(\varM,\kg), \hd_\tla)$ for this graded differential Lie algebra.

Let $\eta$ be a representation of $\kg$ on a vector space $\evE$. Then the trivial fiber bundle $\varE = \varM \times \evE$ carries a natural representation of $\tla(\varM, \kg)$ which we denote by $\eta$ also. The graded differential algebra $(\Omega^\grast(\lieA, \algA(\varE)), \hd_\varE)$ identifies with the total complex of the bigraded algebra $\Omega^\grast(\varM) \otimes \exter^\grast \kg^\ast \otimes \caL(\evE)$ with an obvious differential. For further references, we shall denote by $(\Omega^\grast_\tla(\varM,\kg,\evE), \hd_\evE)$ this differential algebra.

\smallskip
At this stage, it is worth recalling that the underlying local fiber theory makes possible to study locally Lie algebroids as (local) trivial Lie algebroids \cite{Mack05a}. A local description of $\lieA$ over an open subset $U \subset \varM$ is an isomorphism of Lie algebroids $S : \tla(U, \kg) \xrightarrow{\simeq} \lieA_U$.

For a transitive Lie algebroid, a local description $(U,S)$ of $\lieA$ is equivalent to a triple $(U, \Psi, \nabla^0)$ where $\Psi : U \times \kg \xrightarrow{\simeq} \varL_U$ is a local trivialization of $\varL$ as a fiber bundle in Lie algebras and $\nabla^0 :  \Gamma(TU) \rightarrow \lieA_U$ is an injective morphism of Lie algebras and $C^\infty(U)$-modules compatible with the anchors and such that:
\begin{align*}
[\nabla^0_X, \nabla^0_Y] &= \nabla^0_{[X,Y]},
&
\nabla^0_{fX} &= f \nabla^0_X,
&
\rho \circ \nabla^0_X &= X,
\end{align*}
for any $X,Y \in \Gamma(TU)$ and any $f \in C^\infty(U)$. For any $X \in \Gamma(TU)$ and any $\gamma \in \Gamma(U \times \kg)$, one has the compatibility relation $[\nabla^0_X, \iota \circ \Psi(\gamma) ] = \iota \circ \Psi(X \cdotaction \gamma)$. The two local descriptions are related by $S(X \oplus \gamma) = \nabla^0_X + \iota \circ \Psi(\gamma)$.

A transitive Lie algebroid can be completely characterized using a Lie algebroid atlas which is a collection of triples $\{(U_i, \Psi_i, \nabla^{0,i})\}_{i \in I}$ such that $\bigcup_{i \in I} U_i = \varM$ and each triple $(U_i, \Psi_i, \nabla^{0,i})$ is a local trivialization of $\lieA$. The family $(U_i, \Psi_i)$ realizes a system of local trivializations of the fiber bundle in Lie algebras $\varL$. On any $U_{ij} = U_i \cap U_j \neq \ensvide$ one can define $\alpha_{ij} : U_{ij} \rightarrow \Aut(\kg)$ with $\alpha_{ij} = \Psi_i^{-1} \circ \Psi_j$.

Any element $\kX \in \lieA$ can be locally trivialized as $(X_i, \gamma_i)$ over each $U_i$ with $S_i(X_i \oplus \gamma_i) = \kX_{|U_i}$ where $X_i = X_{|U_{i}}$ for $X = \rho(\kX)$. On $U_{ij}$ one has ${X_i}_{|U_{ij}} = {X_j}_{|U_{ij}}$ and 
\begin{equation}
\label{eq-gluingrelationsinnerpartliealgebroidelement}
\gamma_i = \alpha_{ij}(\gamma_j) + \chi_{ij}(X)
\end{equation}
for a well-defined local $1$-form $\chi_{ij} \in \Omega^1(U_{ij}) \otimes \kg$. Finally on $U_{ijk} = U_i \cap U_j \cap U_k \neq \ensvide$, one has the cocycle relations
\begin{align}
\label{eq-changeoftrivializations-compatibilitythreeopensubsets}
\alpha_{ik} &= \alpha_{ij} \circ \alpha_{jk}
&
\chi_{ik} &= \alpha_{ij} \circ \chi_{jk} + \chi_{ij}
\end{align}

These local structures will be used to examine relations between Atiyah Lie algebroids and similar structures in noncommutative geometry.

\subsection{The Atiyah Lie algebroid of a principal bundle}
\label{ex-atiyahliealgebroid}

Let $\varP$ be a $G$-principal bundle over a smooth manifold $\varM$. Let $\pi$ be its projection on $\varM$. Denote by $\raR_g : \varP \rightarrow \varP$, $\raR_g(p) = p \cdotaction g$, the right action of $G$ on $\varP$. Denote by $\kg$ the Lie algebra of the Lie group $G$.

Define the two spaces
\begin{align*}
\Gamma_G(T\varP) &= \{ \sfX \in \Gamma(T\varP) \, / \, \raR_{g\,\ast}\sfX = \sfX \text{ for all } g \in G \}
\\
\Gamma_G(\varP, \kg) &= \{ v : P \rightarrow \kg \, / \, v(p \cdotaction g) = \Ad_{g^{-1}} v(p) \text{ for all } g \in G \}
\end{align*}
They are Lie algebras: the Lie bracket on $\Gamma(T\varP)$ restricts to $\Gamma_G(T\varP)$ and the Lie bracket on $\Gamma(\varP, \kg)$, induced by the Lie bracket on $\kg$, restricts to $\Gamma_G(\varP, \kg)$. They are $C^\infty(\varM)$-modules if one maps $f \in C^\infty(\varM)$ to $\pi^\ast f \in C^\infty(\varP)$.

The first space defines vector fields on $\varP$ which are projectable to the base manifold. We denote by $\pi_\ast : \Gamma_G(T\varP) \rightarrow \Gamma(T\varM)$ the well defined morphism of Lie algebras and $C^\infty(\varM)$-modules. For computational purposes, the property of $\raR$-invariance of a vector field $\sfX \in \Gamma(T\varP)$ is equivalent to the following property about its flow $\varP \times \gR \ni (p,t) \mapsto \phi_\sfX(p,t)$, $\phi_\sfX(p,0) = p$, $\frac{d}{dt} \phi_\sfX(p,t) = \sfX_{|\phi_\sfX(p,t)}$:
\begin{equation*}
\phi_\sfX(p\cdotaction g,t) = \phi_\sfX(p,t) \cdotaction g
\end{equation*}
for any $g \in G$.

The second space $\Gamma_G(\varP, \kg)$ is the space of $(\raR, \Ad)$-equivariant maps $v : P \rightarrow \kg$, which is also the space of sections of the associated vector bundle $\varP \times_\Ad \kg$.

For any $\xi \in \kg$ denote by $\xi^\varP$ the fundamental (vertical) vector field on $\varP$ associated to $\xi$ for the right action $\raR$ of $G$. The map $\iota : \Gamma_G(\varP, \kg) \rightarrow \Gamma_G(T\varP)$ defined by
\begin{equation*}
\iota(v)(p) = -v(p)^\varP_{|p} = \left( \frac{d}{dt} p \cdotaction e^{-t v(p)} \right)_{|t=0}
\end{equation*}
is an injective morphism of Lie algebras and it is $C^\infty(\varM)$-linear.

\begin{definition}
The short exact sequence of Lie algebras and $C^\infty(\varM)$-modules
\begin{equation*}
\xymatrix@1{{\algzero} \ar[r] & {\Gamma_G(\varP, \kg)} \ar[r]^-{\iota} & {\Gamma_G(T\varP)} \ar[r]^-{\pi_\ast} & {\Gamma(T\varM)} \ar[r] & {\algzero}}
\end{equation*}
defines $\Gamma_G(T\varP)$ as a transitive Lie algebroid over $\varM$. This is the well-known Atiyah Lie algebroid associated to $\varP$ \cite{MR0086359}.
\end{definition}

When $\varP = \varM \times G$ is trivial, one has $\Gamma_G(T\varP) = \tla(\varM, \kg)$ as mentioned earlier.

We shall be interested in the description of the space of forms on the Lie algebroid $\Gamma_G(T\varP)$ with values in its kernel $\Gamma_G(\varP, \kg)$. To simplify notations, we denote it by $(\Omega^\grast_\lie(\varP, \kg), \hd)$.
For any $\sfX \in \Gamma_G(T\varP)$ and any $v \in \Gamma_G(\varP, \kg)$, one has $[\sfX, \iota(v)] = \iota(\sfX \cdotaction v)$ where $\sfX$ acts as a vector field on the (equivariant) function $v : \varP \rightarrow \kg$. This gives the explicit form of the differential $\hd$ on $\Omega^\grast_\lie(\varP, \kg)$:
\begin{multline}
\label{eq-differentialAtiyah}
(\hd \omega)(\sfX_1, \dots, \sfX_{p+1}) = \sum_{i=1}^{p+1} (-1)^{i+1} \sfX_i \cdotaction \omega(\sfX_1, \dots \omi{i} \dots, \sfX_{p+1})\\
+ \sum_{1 \leq i < j \leq p+1} (-1)^{i+j} \omega([\sfX_i, \sfX_j], \sfX_1, \dots \omi{i} \dots \omi{j} \dots, \sfX_{p+1})
\end{multline}

From now on, we assume that the structure group $G$ is connected and simply connected.

Then a map $v : \varP \rightarrow \kg$ is in $\Gamma_G(\varP, \kg)$ if and only if $\xi^\varP \cdotaction v + [\xi, v] = 0$ for any $\xi \in \kg$. Let us introduce the following subspace of $\Gamma(T\varP \oplus (\varP \times \kg)) = \tla(\varP, \kg)$:
\begin{equation*}
\kg_\equ = \{ \xi^\varP \oplus \xi \ / \ \xi \in \kg \}\, .
\end{equation*}

The Lie algebra $\kg_\equ$ defines a natural Cartan operation on the differential complex $(\Omega^\grast_\tla(\varP,\kg), \hd_\tla)$. Let us denote by $(\Omega^\grast_\tla(\varP,\kg)_{\kg_\equ}, \hd_\tla)$ the differential graded subcomplex of basic elements. The main result of this sub-section is the following identification.

\begin{proposition}
\label{prop-identificationdifferentialcalculusAtiyah}
$(\Omega^\grast_\lie(\varP, \kg), \hd)$ and $(\Omega^\grast_\tla(\varP,\kg)_{\kg_\equ}, \hd_\tla)$ are isomorphic as differential graded complexes.
\end{proposition}

In \cite{Mack05a} it is shown that the differential complex $(\Omega^\grast_\lie(\varP, \kg), \hd)$ is isomorphic to the differential complex of $(\raR, \Ad)$-equivariant forms in $(\Omega^\grast(\varP) \otimes \kg, d)$, which ought to correspond to ${\kg_\equ}$-invariant forms in our terminology.

\smallskip
In order to prove Proposition~\ref{prop-identificationdifferentialcalculusAtiyah}, we need the following technical constructions and preliminary results.

Let $\caZ$ be defined as the $C^\infty(\varP)$-module generated by $\kg_\equ$, and let define $\caN = \Gamma_G(T\varP) \oplus \caZ$. Notice that $\caN$ is a $C^\infty(\varM)$-module but not a $C^\infty(\varP)$-module.

We define the $C^\infty(\varP)$-linear map $\Gamma(\varP \times \kg) \rightarrow \Gamma(T\varP)$, written as $\gamma \mapsto \gamma^\varP$, by the following relation. For any $p \in \varP$, 
\begin{equation}
\label{eq-defgeneralizedfundamentalvectorfield}
\gamma^\varP(p) = \left(\frac{d}{dt} p\cdotaction e^{t \gamma(p)}\right)_{|t=0}
\end{equation}
Notice that $(t,p) \mapsto p\cdotaction e^{t \gamma(p)}$ is not the flow of the vector field $\gamma^\varP$, except in two situations: for a constant $\gamma \in \kg$ this corresponds to the usual definition of fundamental (vertical) vector fields, and for $\gamma \in \Gamma_G(\varP, \kg)$ this corresponds to the vector field $-\iota(\gamma)$.

By definition, any element $z \in \caZ$ is a sum of elements of the type $f (\xi^\varP \oplus \xi)$ where $f \in C^\infty(\varP)$ and $\xi \in \kg$. Using the $C^\infty(\varP)$-linearity of $\gamma \mapsto \gamma^\varP$, one shows that $z = \gamma^\varP \oplus \gamma$ where $\gamma \in \Gamma(\varP \times \kg)$ is the projection of $z \in \caZ \subset \tla(\varP, \kg)$ onto $\Gamma(\varP \times \kg)$.

\begin{proposition}
\begin{enumerate}
\item For any $v \in \Gamma_G(\varP, \kg)$ and any $\gamma^\varP \oplus \gamma \in \caZ$, one has $\gamma^\varP \cdotaction v + [\gamma, v] = 0$.

\item Let $X \oplus \gamma \in \tla(\varP, \kg)$. If for any $\xi \in \kg$ one has $[X \oplus \gamma, \xi^\varP \oplus \xi] \in \caZ$ then $[X \oplus \gamma, \caZ] \subset \caZ$.

\item $\caZ$ is a Lie ideal and a $C^\infty(\varM)$-sub module in $\caN$.
\end{enumerate}
\end{proposition}

\begin{proof}
\begin{enumerate}
\item For any $f \in C^\infty(\varP)$ and any $\xi \in \kg$, one has $(f \xi^\varP)\cdotaction v = f (\xi^\varP \cdotaction v) = -f [\xi, v] = -[f \xi, v]$, which shows the relation is true on the generators of $\caZ$ of the type $f (\xi^\varP \oplus \xi)$.

\item It is sufficient to establish the fact on an element $f (\xi^\varP \oplus \xi)$ where $f \in C^\infty(\varP)$ and $\xi \in \kg$. One has $[X \oplus \gamma, f (\xi^\varP \oplus \xi)] = (X \cdotaction f) (\xi^\varP \oplus \xi) + f [X \oplus \gamma, \xi^\varP \oplus \xi]$. This is obviously in $\caZ$ if $[X \oplus \gamma, \xi^\varP \oplus \xi] \in \caZ$.

\item With $X \oplus \gamma \in \kg_\equ$ in the previous relation, one shows that $\caZ$ is a Lie algebra. For any $\sfX \in \Gamma_G(T\varP)$ and any $\xi^\varP \oplus \xi \in \kg_\equ$ with $\xi \in \kg$, the Lie bracket $[\sfX \oplus 0, \xi^\varP \oplus \xi]$ in $\tla(\varP, \kg)$ takes the form $[\sfX, \xi^\varP] \oplus 0$. Recall that the flow $\phi_{\sfX,t}$ of $\sfX$ satisfies $\phi_{\sfX,t}(p \cdotaction g) = \phi_{\sfX,t}(p) \cdotaction g$ for any $g \in G$ and that the flow of $\xi^\varP$ is $\phi_{\xi, s}(p) = p \cdotaction e^{s \xi}$. Then one has
\begin{align*}
[\xi^\varP, \sfX](p) &= \left(\frac{d}{dt} \left( \frac{d}{ds}  \phi_{\sfX,t}\left( p \cdotaction e^{s \xi} \right)\cdotaction e^{-s \xi} \right)_{|s=0}\right)_{|t=0} \\
&= \left(\frac{d}{dt} \left( \phi_{\sfX,t}( p ) \right)_{|s=0}\right)_{|t=0} = 0 
\end{align*}
By the previous result, this implies that $[\sfX \oplus 0, \xi^\varP \oplus \xi] \in \caZ$ for any $\xi \in \kg$, so that $[\Gamma_G(T\varP), \caZ] \subset \caZ$. This in turn implies that $[\caN, \caZ] \subset \caZ$.
\end{enumerate}
\end{proof}

Notice that in the proof we have shown that $[\Gamma_G(T\varP), \kg_\equ] = 0$.

As a consequence of this proposition, one has the splitting of the short exact sequence of Lie algebroids over $\varM$
\begin{equation*}
\xymatrix@1{{\algzero} \ar[r] & {\caZ} \ar[r] & {\caN} \ar[r]^-{\rho_\varP} & {\Gamma_G(T\varP)} \ar[r] & {\algzero}}
\end{equation*}

\begin{proposition}
\label{prop-localfamilyofgeneratorsAtiyah}
Let $U \subset \varM$ be an open subset which trivializes $\varP$ and which is also a open subset for a chart on $\varM$. Then there exists a family $\{\sfX^i\}_{i=1, \dots, \dim \varP}$ of right invariant vector fields in $\Gamma_G(T\varP_{|U})$ such that:
\begin{enumerate}
\item For any $p \in \varP_{|U}$, $\{\sfX^i(p)\}$ is a basis of $T_p\varP$.

\item This family generates $\Gamma_G(T\varP_{|U})$ as a $C^\infty(U)$-module and the decomposition $\sfX = f_i \sfX^i \in \Gamma_G(T\varP_{|U})$ defines some unique right invariant functions $f_i \in C^\infty(\varP_{|U})$ (so that $f_i \in C^\infty(U)$).

\item This family generates $\Gamma(T\varP_{|U})$ as a $C^\infty(\varP_{|U})$-module and the decomposition $X = g_i \sfX^i \in \Gamma(T\varP_{|U})$ defines some unique functions $g_i \in C^\infty(\varP_{|U})$. 
\end{enumerate}

Two such families differ by linear combinations whose coefficients are functions in $C^\infty(U)$.

\end{proposition}

We will call such a family a local family of generators in $\Gamma_G(T\varP)$.

\begin{proof}
In the local trivialization $\varP_{|U} \simeq U \times G$, denote by $x=(x^\mu)$ some local coordinates on $U$, so that we can write $p \simeq (x,g) \in U \times G$ with $g \in G$. One has $T_p \varP \simeq T_{x} U \oplus T_{g} G$. Denote by $\{e_a\}$ a basis of $\kg$, and consider the $e_a$'s as left invariant vector fields on $G$ so that they also define a basis of $T_gG$ at each point $g \in G$.

Consider a particular smooth vector field $x \mapsto X(x) = X^\mu(x) \partial_\mu + \overline{X}^a(x) e_a \in T_{x} U \oplus T_{e} G$ whose coordinate functions do not depend on $g \in G$. Then, applying the right translation $T_{(x,e)}\raR_{g}$ on these vectors, one gets a local right invariant vector fields $(x,g) \mapsto \sfX(x,g) = X^\mu(x, g) \partial_\mu + \overline{X}^a(x, g) e_a \in T_x U \oplus T_{g} G$. By construction, one has $X^\mu(x, g) = X^\mu(x)$ and $\overline{X}^a(x, g) = A^a_b(g) \overline{X}^b(x)$ where $A^a_b(g) e_a = \Ad_{g^{-1}} e_b$, so that $\sfX(x,g) = X^\mu(x) \partial_\mu + A^a_b(g) \overline{X}^b(x) e_a$.

Let us now begin with a smooth family $x \mapsto X^i(x) \in T_{x} U \oplus T_{e} G$ of vector fields as before, for $i=1, \dots, \dim \varP$, which defines a basis of the vector space $T_{x} U \oplus T_{e} G$ for each $x \in U$. Then the family of right invariant extensions $(x,g) \mapsto \sfX^i(x,g) \in T_{x} U \oplus T_{g} G$ defines also a basis of the vector spaces $T_{x} U \oplus T_{g} G$ for each $(x,g) \in U \times G$. This is point~1.

Any right invariant vector field $\sfX \in \Gamma_G(T\varP_{|U})$ can be written as $\sfX(x,g) = f_i(x,g) \sfX^i(x,g) \in T_{x} U \oplus T_{g} G$. Right invariance implies that $f_i(x,g) = f_i(x,g')$ for any $g,g' \in G$, so that $\sfX(x,g) = f_i(x) \sfX^i(x,g)$. The functions $f_i \in C^\infty(U)$ are uniquely determined because the $\sfX^i(x,g)$'s form a basis at each point. This is point~2.

Point~3 is proved similarly.

Let us assume given a second family $\sfY^i \in \Gamma_G(T\varP_{|U})$ of right invariant vector fields with the same properties. Then using point~2, there are unique functions $\alpha^i_j \in C^\infty(U)$ such that $\sfY^i = \alpha^i_j \sfX^j$.
\end{proof}

\begin{corollary}
\label{cor-rightinvariantvectorfieldsandNgeneratestrivialLiealgebroidAtiyah}
$\Gamma_G(T\varP)$ generates the space $\Gamma(T\varP)$ as a $C^\infty(\varP)$-module and $\caN$ generates the space $\tla(\varP, \kg)$ as a $C^\infty(\varP)$-module.
\end{corollary}

\begin{proof}
Using the structure of $C^\infty(\varM)$-module of $\Gamma(T\varP)$ and a partition of unity subordinated to a covering of $\varM$ by open subsets which both trivialize $\varP$ and are open subsets for charts on $\varM$, the problem is to show that $\Gamma_G(T\varP_{|U})$ generates $\Gamma(T\varP_{|U})$ as a $C^\infty(U)$-module for such any open subset $U \subset \varM$. This is a direct consequence of point~3 of Proposition~\ref{prop-localfamilyofgeneratorsAtiyah}.

Any element $X \oplus \gamma \in \tla(\varP, \kg)$ can be written as $(X - \gamma^\varP) \oplus 0 + (\gamma^\varP \oplus \gamma)$ where $\gamma^\varP \oplus \gamma \in \caZ$. Thus, it remains to show that the vector field $X - \gamma^\varP$ is in the $C^\infty(\varP)$-module generated by $\Gamma_G(T\varP)$, which is indeed the case by the previous result.
\end{proof}

\begin{proof}[Proof of Proposition~\ref{prop-identificationdifferentialcalculusAtiyah}]
Recall that $\Omega^\grast_\tla(\varP,\kg)$ is the space of forms from $\tla(\varP, \kg)$ to $\kg$. Let us define a canonical map
\begin{equation*}
\lambda : \Omega^\grast_\tla(\varP,\kg)_{\kg_\equ} \rightarrow \Omega^\grast_\lie(\varP, \kg).
\end{equation*}
For any $\widehat{\omega} \in \Omega^r_\tla(\varP,\kg)_{\kg_\equ}$, we define $\alpha = \lambda(\widehat{\omega}) \in \Omega^r_\lie(\varP, \kg)$ as follows. To any $\sfX_1, \dots, \sfX_r \in \Gamma_G(T\varP)$, we associate any $\widehat{X}_1, \dots, \widehat{X}_r \in \caN$ such that $\rho_\varP(\widehat{X}_i) = \sfX_i$. Then 
\begin{equation*}
\alpha(\sfX_1, \dots, \sfX_r)(p) = \widehat{\omega}(\widehat{X}_1, \dots, \widehat{X}_r)(p) \in \kg
\end{equation*}
is well defined because, as a $\kg_\equ$-horizontal form, $\widehat{\omega}$ vanishes on $\kg_\equ$, and so on $\caZ$. For any $\xi^\varP \oplus \xi \in \kg_\equ$, the Lie derivative in the direction $\xi^\varP \oplus \xi$ takes the explicit form
\begin{multline}
\label{eq-invarianceomegahatAtiyah}
(L_{\xi^\varP \oplus \xi} \widehat{\omega})(\widehat{X}_1, \dots, \widehat{X}_r) = \xi^\varP \cdotaction \widehat{\omega}(\widehat{X}_1, \dots, \widehat{X}_r) + [\xi, \widehat{\omega}(\widehat{X}_1, \dots, \widehat{X}_r)]\\
- \sum_{i=1}^{r} \widehat{\omega}(\widehat{X}_1, \dots, [ \xi^\varP \oplus \xi, \widehat{X}_i], \dots, \widehat{X}_r)
\end{multline}
The commutators $[ \xi^\varP \oplus \xi, \widehat{X}_i]$ are in $\caZ$ because $\widehat{X}_i \in \caN$ and $\xi^\varP \oplus \xi \in \caZ$, so that each term in the last sum is zero. The $\kg_\equ$-invariance of $\widehat{\omega}$ then implies $0 = L_{\xi^\varP \oplus \xi} \widehat{\omega} = \xi^\varP \cdotaction \widehat{\omega}(\widehat{X}_1, \dots, \widehat{X}_r) + [\xi, \widehat{\omega}(\widehat{X}_1, \dots, \widehat{X}_r)]$, which means that $\widehat{\omega}(\widehat{X}_1, \dots, \widehat{X}_r) \in \Gamma_G(\varP, \kg)$. Then $\alpha \in \Omega^\grast_\lie(\varP, \kg)$.

When applied to invariant forms and to elements in $\caN = \Gamma_G(T\varP) \oplus \caZ$, the differential $\hd_\tla = \dd + \ds'$ reduces to the differential $\hd$ given by \eqref{eq-differentialAtiyah}. The defined map is then a morphism of differential graded complexes.

\medskip
The injectivity of $\lambda$ is a consequence of Corollary~\ref{cor-rightinvariantvectorfieldsandNgeneratestrivialLiealgebroidAtiyah}: if $\alpha(\sfX_1, \dots, \sfX_r) = 0$ for any $\sfX_1, \dots, \sfX_r \in \Gamma_G(T\varP)$, then $\widehat{\omega}(\widehat{X}_1, \dots, \widehat{X}_r) = \alpha(\rho_\varP(\widehat{X}_1), \dots, \rho_\varP(\widehat{X}_r)) = 0$ for any $\widehat{X}_1, \dots, \widehat{X}_r \in \caN$, so that $\widehat{\omega} = 0$ by $C^\infty(\varP)$-linearity.

Let us show that $\lambda$ is surjective. It is sufficient to work over a subset $U \subset \varM$ which trivializes $\varP$. Consider a local family $\{\sfX^i\}_{i=1, \dots, \dim \varP}$ of generators in $\Gamma_G(T\varP)$ over $U$. For any $\alpha \in \Omega^r_\lie(\varP_{|U}, \kg)$, let $\widehat{\omega}$ be the $C^\infty(U)$-multilinear antisymmetric map defined on $\caN^r_{|U}$ with values in maps $\varP_{|U} \rightarrow \kg$ given by
\begin{equation*}
\widehat{\omega}(\sfX^{i_1} + z_1, \dots, \sfX^{i_r} + z_r) = \alpha(\sfX^{i_1}, \dots, \sfX^{i_r})
\end{equation*}
for any $z_k \in \caZ_{|U}$. Because two local families of generators in $\Gamma_G(T\varP_{|U})$ are related by functions in $C^\infty(U)$ and because $\alpha$ is $C^\infty(U)$-linear, this definition makes sense and does not depend on the choice of the local family of generators in $\Gamma_G(T\varP)$. 

Using the uniqueness of the decomposition $X = g_i \sfX^i \in \Gamma(T\varP_{|U})$ we can impose $C^\infty(\varP_{|U})$-linearity on $\widehat{\omega}$ in order to extend $\widehat{\omega}$ as a $C^\infty(\varP_{|U})$-multilinear antisymmetric map defined on $(\Gamma(T\varP_{|U} \oplus (U \times \kg))^r = \tla(\varP_{|U}, \kg)^r$, so that $\widehat{\omega} \in \Omega^r_\tla(\varP_{|U},\kg)$. Once again, note that this extension is well defined because changing the local families of generators in $\Gamma_G(T\varP_{|U})$ requires only coefficients in $C^\infty(U)$ for which linearity is already given. By construction $\widehat{\omega}$ is $\kg_\equ$-horizontal. Using the fact that $\alpha(\sfX^{i_1}, \dots, \sfX^{i_r}) \in \Gamma_G(\varP_{|U}, \kg)$, it is easy to show, using \eqref{eq-invarianceomegahatAtiyah}, that $\widehat{\omega}$ is $\kg_\equ$-invariant. This means that $\widehat{\omega} \in \Omega^r_\tla(\varP_{|U},\kg)_{\kg_\equ}$.
\end{proof}

\smallskip
Let us now consider connections of the Atiyah Lie algebroid. It is well-known that connections on the Atiyah Lie algebroid are exactly connections on the principal fiber bundle from which this Lie algebroid is constructed. The splitting $\nabla : \Gamma(T\varM) \rightarrow \Gamma_G(T\varP)$ corresponds to the horizontal lift $X \mapsto X^h$ defined by the connection $1$-form $\omega$ on the principal fiber bundle.

In terms of differential forms, the connection $1$-form $\omega$ on $\varP$ is an element of $\Omega^1(\varP) \otimes \kg$ which is $(\raR, \Ad)$-equivariant and normalized on vertical vector fields by $\omega(\xi^\varP) = \xi$ for any $\xi \in \kg$. We consider $\omega$ as an element of $\Omega^1(\varP) \otimes \exter^0 \kg^\ast \otimes \kg$. The $(\raR, \Ad)$-equivariance infinitesimally reads $0 = L_{\xi^\varP \oplus \xi} \omega = L_{\xi^\varP} \omega + [\xi, \omega]$ for any $\xi^\varP \oplus \xi \in \kg_\equ$.

Let $\theta \in \exter^1 \kg^\ast \otimes \kg$ be the Maurer-Cartan $1$-form on $G$. Recall that $\theta(\xi) = \xi$ for any $\xi \in \kg$. We can extend $\theta$ by $C^\infty(\varP)$-linearity to an element $\theta \in C^\infty(\varP) \otimes \exter^1 \kg^\ast \otimes \kg$ so that $\theta(\gamma) = \gamma$ for any $\gamma : \varP \rightarrow \kg$. Then, for any $\xi \in \kg$ and $\gamma : \varP \rightarrow \kg$, one has $(L_{\xi^\varP \oplus \xi}\theta)(\gamma) = \xi^\varP \cdotaction \theta(\gamma) + [\xi, \theta(\gamma)] - \theta([\xi^\varP + \xi, \gamma]) = \xi^\varP \cdotaction \gamma + [\xi, \gamma] - (\xi^\varP \cdotaction \gamma + [\xi, \gamma]) = 0$.

It is now easy to show that the $1$-form $\widehat{\omega} = \omega - \theta \in \Omega^1_\tla(\varP,\kg)$ is $\kg_\equ$-basic, and that
\begin{equation}
\label{eq-basicconnectiononeform}
\widehat{\omega} = \omega - \theta \in \Omega^1_\tla(\varP,\kg)_{\kg_\equ}
\end{equation}
corresponds by Proposition~\ref{prop-identificationdifferentialcalculusAtiyah} to the $1$-form $\alpha \in \Omega^1_\lie(\varP, \kg)$ associated in Proposition~\ref{connectiononeformanscurvatureinalgebraicterms} to the connection $\nabla$. Conversely, Proposition~\ref{prop-identificationdifferentialcalculusAtiyah} associates to a $1$-form $\alpha \in \Omega^1_\lie(\varP, \kg)$, such that $\alpha (\iota(v)) = -v$, a $1$-form $\widehat{\omega} \in \Omega^1_\tla(\varP,\kg)$ which is basic for the Cartan operation of $\kg_\equ$. Necessarily one has $\widehat{\omega} = \omega - \theta$ where $\omega \in \Omega^1(\varP) \otimes \kg$ has the properties of an ordinary connection $1$-form on $\varP$.

The curvature of $\alpha$ as a Lie algebroid connection is given by the expression $\hR(\sfX, \sfY) = (\hd \alpha)(\sfX, \sfY) + [\alpha(\sfX), \alpha(\sfY)]$ for any $\sfX, \sfY \in \Gamma_G(T\varP)$. For any $X \oplus \gamma, Y \oplus \eta \in \tla(\varP, \kg)$, a straightforward computation shows that
\begin{multline*}
(\hd_\tla \widehat{\omega})(X \oplus \gamma, Y \oplus \eta) + [\widehat{\omega}(X \oplus \gamma), \widehat{\omega}(Y \oplus \eta)]\\
= (d \omega)(X, Y) + [\omega(X), \omega(Y)] = \Omega(X,Y)
\end{multline*}
where $\Omega$ is the curvature of the connection $\omega$. Using the usual properties of the curvature $2$-form $\Omega \in \Omega^2(\varP) \otimes \kg$, it is readily shown that as an element in $\Omega^2_\tla(\varP,\kg)$ it is $\kg_\equ$-basic. By Proposition~\ref{prop-identificationdifferentialcalculusAtiyah} and the previous relations, $\Omega \in \Omega^2_\tla(\varP,\kg)_{\kg_\equ}$ corresponds to $\hR \in \Omega^2_\lie(\varP, \kg)$.

\medskip
In the same way, to any generalized connection $1$-form $\alpha$ on $\Gamma_G(T\varP)$ in the sense of Definition~\ref{def-generalizedconnections}, one can associate by Proposition~\ref{prop-identificationdifferentialcalculusAtiyah} a $1$-form 
\begin{equation}
\label{eq-basicgeneralizedconnectiononeform}
\widehat{\omega} =\omega + \phi \in \Omega^1_\tla(\varP,\kg) = (\Omega^1(\varP) \otimes \kg) \oplus (C^\infty(\varP) \otimes \exter^1 \kg^\ast \otimes \kg)
\end{equation}
which is basic for the Cartan operation of $\kg_\equ$. Note that $\omega$ and $\phi$ are $\kg_\equ$-invariant, but $\omega$ is not necessarily a connection $1$-form on $\varP$ and $\phi$ is not necessarily related to the Maurer-Cartan form on $G$.

\medskip
Let $\evE$ be a vector space and $\ell_g : \evE \rightarrow \evE$ be a linear representation of the structure group $G$. Denote by $\varE = \varP \times_\ell \evE$ the associated vector bundle. Recall that the space of sections of $\varE$ is given by
\begin{equation*}
\Gamma(\varE) = \{ s : \varP \rightarrow \evE \, / \, s(p \cdotaction g) = \ell_{g^{-1}} s(p) \text{ for all } g \in G \}.
\end{equation*}
$\varE$ carries a natural representation of the Lie algebroid $\Gamma_G(T\varP)$ given by
\begin{align*}
\phi : \Gamma_G(T\varP) &\rightarrow \kD(\varE)
&
\phi(\sfX)(s) &= \sfX \cdotaction s
\end{align*}

Denote by $\End(\evE)$ the space of linear operators on the vector space $\evE$ and by $\eta$ the representation of the Lie algebra $\kg$ on $\evE$ induced by $\ell$, so that $\eta(v) \in \End(\evE)$ for any $v \in \kg$. The kernel $\algA(\varE)$ of $\kD(\varE)$ as a transitive Lie algebroid is
\begin{equation*}
\algA(\varE) = \{ a : \varP \rightarrow \End(\evE) \, / \, a(p \cdotaction g) = \ell_{g^{-1}} \circ a(p) \circ \ell_g \text{ for all } g \in G \}
\end{equation*}
and the $\phi$ applied to $v \in \Gamma_G(\varP, \kg)$ is just $\eta(v) : \varP \rightarrow \End(\evE)$. The induced representation $\widetilde{\phi}$ of $\Gamma_G(T\varP)$ on $\End(\varE)$ is given by $\widetilde{\phi}(\sfX) \cdotaction a = \sfX \cdotaction a$.

The Lie algebra $\kg_\equ$ defines a Cartan operation on the graded differential algebra $(\Omega^\grast_\tla(\varM,\kg,\evE), \hd_\evE)$ introduced in sub-section~\ref{trivialliealgebroid}. We denote by $(\Omega^\grast_\tla(\varM,\kg,\evE)_{\kg_\equ}, \hd_\evE)$ the graded differential sub-algebra of basic elements for this operation. By adapting the proof of Proposition~\ref{prop-identificationdifferentialcalculusAtiyah}, one shows that

\begin{proposition}
\label{prop-identificationdifferentialcalculusrepresentation}
$(\Omega^\grast(\Gamma_G(T\varP), \algA(\varE)), \hd_\varE)$ and $(\Omega^\grast_\tla(\varP,\kg,\evE)_{\kg_\equ}, \hd_\evE)$ are isomorphic as differential graded algebras. 
\end{proposition}

This identification permits to associate to any $\Gamma_G(T\varP)$-connection $\hnabla^\varE$ on $\varE$ a $1$-form $\widehat{\omega} \in \Omega^1_\tla(\varP,\kg,\evE)$ which is basic for the Cartan operation of $\kg_\equ$.

\medskip
As already recalled, it is well known that $\Gamma_G(\varP, \kg)$ is the Lie algebra of the gauge group of the principal fiber bundle $\varP$. The notion of infinitesimal gauge transformations given by Definition~\ref{def-infinitesimalgaugetransformations} and the corresponding action on connections given by Definition~\ref{def-infinitesimalgaugetransformationofaconnection} (and its explicit expression on the connection $1$-forms given by Proposition~\ref{prop-infinitesimalgaugetransformationontheconnectiononeform}) turn out to be generalizations of those obtained in the context of the Atiyah Lie algebroid. For Atiyah Lie algebroids, one can define gauge transformations as vertical automorphisms of the principal fiber bundle $\varP$, so that the gauge group turns out to be well defined in this case.

\smallskip
Local trivializations of the Lie algebroid $\Gamma_G(T\varP)$ are constructed using local trivializations of the principal fiber bundle $\varP$. Let $\{(U_i, \varphi_i)\}$ be a system of trivializations of $\varP$, with $\varphi_i : \varP_{|U_i} \xrightarrow{\simeq} U_i \times G$ and let $\phi_i=\varphi_i^{-1}$. Any $\sfX \in \Gamma_G(T\varP)$ is decomposed over $U_i$ into two pieces: $X = \pi_\ast(\sfX) \in \Gamma(U)$ and $\gamma_i : U_i \rightarrow \kg$. The explicit expressions for the triples $(U_i, \Psi_i, \nabla^{0, i})$ are given by:
\begin{align*}
(\nabla^{0, i}_X)(p) &= ({\phi_i}_\ast X)(p)
&
\Psi_i(\gamma)\circ \phi_i(x,g) &= \Ad_{g^{-1}} \gamma(x) 
\end{align*}

On $U_{ij} = U_i \cap U_j \neq \ensvide$, denote by $g_{ij} : U_{ij} \rightarrow G$ the transition functions for $\varP$. Then using the explicit expression for $\iota$, one gets $\gamma_i = g_{ij} \gamma_j g_{ij}^{-1} + g_{ij} d g_{ij}^{-1}(X)$, so that
\begin{align}
\label{eq-alphachiatiyah}
\alpha_{ij}(\gamma) &= g_{ij} \gamma g_{ij}^{-1}
&
\chi_{ij}(X) &= g_{ij} d g_{ij}^{-1}(X)
\end{align}
All the relations which are required to be satisfied both by the $\alpha_{ij}$'s and the $\chi_{ij}$'s can be easily established using some standard computations performed in the theory of fiber bundles.

This description of $\Gamma_G(T\varP)$ in terms of local trivializations will be used in the next sub-section.

\subsection{The Lie algebroid of the derivations of the endomorphism algebra}
\label{ex-derivationsendomorphismalgebraliealgebroid}

The Lie algebroid described in this sub-section takes its roots into a particular situation well studied in the context of the derivations based noncommutative geometry. This example will help us to understand the relationship between the generalized notions of connections on Lie algebroids and in noncommutative geometry.

Let $\varE$ be a $SL(n, \gC)$-vector bundle over the manifold $\varM$ with fiber $\gC^n$. Consider $\End(\varE)$, the fiber bundle of endomorphisms of $\varE$. We shortly denote by $\algA$ the algebra of sections of $\End(\varE)$ (previously denoted by $\algA(\varE)$). This is the algebra of the endomorphism algebra of a $SL(n, \gC)$-vector bundle.

This algebra has been studied from the point of view of its noncommutative geometry based on derivations \cite{Mass14,Mass15}. In Appendix~\ref{briefreviewofsomenoncommutativestructures}, a brief review is given for the essential structures introduced according to this noncommutative geometry. The noncommutative geometry of the endomorphism algebra $\algA$ presented here has been summarized in \cite{Mass30}. Here we collect only the  necessary results to render explicit the relationship with the notion of connection on Lie algebroids.

Let us denote by $\der(\algA)$ the Lie algebra of derivations of $\algA$. This is also a module over the center of $\algA$, which is $C^\infty(\varM)$. By the $\ad$-representation the ideal of inner derivations $\Int(\algA)$ identifies with $\algA_0$, the traceless elements in $\algA$. Define $\rho : \der(\algA) \rightarrow \der(\algA)/\Int(\algA) = \Out(\algA)$ the quotient map. This map can be explicitly identified with the restriction of any $\kX \in \der(\algA)$ to the center. So that $\rho : \der(\algA) \rightarrow \Gamma(T\varM)$ is both a morphism of Lie algebras and of $C^\infty(\varM)$-modules.

The short exact sequence of Lie algebras and $C^\infty(\varM)$-modules
\begin{equation}
\label{eq-sesderivationsendomorphismalgebra}
\xymatrix@1{{\algzero} \ar[r] & {\algA_0} \ar[r]^-{\ad} & {\der(\algA)} \ar[r]^-{\rho} & {\Gamma(T\varM)} \ar[r] & {\algzero}}
\end{equation}
defines $\der(\algA)$ as a transitive Lie algebroid over $\varM$, with $\iota = \ad$ \cite{Mass14}.

Let $\varP$ be the principal $SL(n, \gC)$-fiber bundle to which $\varE$ (and then $\End(\varE)$) is associated and consider its Atiyah Lie algebroid $\Gamma_G(T\varP)$.

Let $\{(U_i, \varphi_i)\}$ be a system of local trivializations of $\varP$, with $\varphi_i : \varP_{|U_i} \xrightarrow{\simeq} U_i \times SL(n, \gC)$. This also defines a system of trivializations of $\End(\varE)$. Denote by $\Psi_i : \algA_{|U_i} \xrightarrow{\simeq} C^\infty(U_i) \otimes M_n(\gC)$ the associated trivialization of $\End(\varE)$ induced by $(U_i, \varphi_i)$. Let $\theta$ be the Maurer-Cartan form on $SL(n, \gC)$. Let $\pi_2 : U_i \times SL(n, \gC) \rightarrow SL(n, \gC)$ be the projection on the second factor. Then $\omega_i = \varphi_i^\ast \pi_2^\ast \theta$ is a local flat connection on $\varP_{|U_i}$. This flat connection defines a covariant derivative $\nabla^{0, i}$ on $\End(\varE_{|U_i})$, which is also a derivation of $\algA_{|U_i}$. Then the triples $(U_i, \Psi_i, \nabla^{0, i})$ define an Lie algebroid atlas for $\der(\algA)$.

On $U_{ij} = U_i \cap U_j \neq \ensvide$, denote by $g_{ij} : U_{ij} \rightarrow SL(n, \gC)$ the transition functions for $\varP$. Then one has
\begin{align*}
\alpha_{ij}(\gamma) &= g_{ij} \gamma g_{ij}^{-1}
&
\chi_{ij}(X) &= g_{ij} (X \cdotaction g_{ij}^{-1}).
\end{align*}

Any element $\kX \in \der(\algA)$ can be written locally as $\kX_i = X \oplus \gamma_i$, with $\tr(\gamma_i) = 0$. The gluing relations of these trivializations on $U_{ij} = U_i \cap U_j \neq \ensvide$ are: $\gamma_i = g_{ij} \gamma_j g_{ij}^{-1} + g_{ij} (X \cdotaction g_{ij}^{-1})$. Comparing these local trivializations and the ones obtained in \eqref{eq-alphachiatiyah} for $\Gamma_G(T\varP)$, one gets:

\begin{proposition}
\label{prop-identificationLiealgebroidsDerivationsAtiyah}
The Lie algebroid $\der(\algA)$ is isomorphic to the Atiyah Lie algebroid $\Gamma_G(T\varP)$ of the $SL(n, \gC)$-principal fiber bundle.
\end{proposition}

In both cases, the local descriptions of an element of these Lie algebroids is given by $(X, \gamma)$ on $U \subset \varM$. For $\der(\algA)$, $\gamma$ is promoted to an inner derivation $\ad_\gamma$ acting on $\algA_{|U}$ and for $\Gamma_G(T\varP)$, $\gamma$ is promoted to a vertical vector field in $\Gamma_G(T\varP_{|U})$.

Denote by $(\Omega^\grast_\lie(\algA, \algA_0), \hd)$ the differential calculus $(\Omega^\grast(\der(\algA), \algA_0), \hd)$ defined on the transitive Lie algebroid $\der(\algA)$.
Denote by $(\Omega^\grast_\der(\algA), \hdNC)$ the differential calculus defined in Definitions~\ref{def-thegradeddifferentialalgebraNCunderline}.

The natural representation of $\der(\algA)$ on $\algA_0$ defining the differential $\hd$ on the complex $\Omega^\grast_\lie(\algA, \algA_0)$ in Definition~\ref{def-formsvalueskernel} is given in this situation by $(\kX, a) \mapsto \kX \cdotaction a \in \algA_0$. The differential $\hd$ then takes the explicit expression
\begin{multline*}
(\hd \omega)(\kX_1, \dots, \kX_{p+1}) = \sum_{i=1}^{p+1} (-1)^{i+1} \kX_i \cdotaction \omega(\kX_1, \dots \omi{i} \dots, \kX_{p+1})\\
+ \sum_{1 \leq i < j \leq p+1} (-1)^{i+j} \omega([\kX_i, \kX_j], \kX_1, \dots \omi{i} \dots \omi{j} \dots, \kX_{p+1})
\end{multline*}
for any $\omega \in \Omega^p_\lie(\algA, \algA_0)$. A direct comparison with \eqref{eq-differentialNC} and with the definition of $\Omega^\grast_\der(\algA)$ shows the following:

\begin{proposition}
\label{prop-relationsbetweendifferentialcalculi}
The differential calculus $(\Omega^\grast_\lie(\algA, \algA_0), \hd)$ constructed from the Lie algebroid structure is included in the derivation-based differential calculus $(\Omega^\grast_\der(\algA), \hdNC)$ as traceless forms.
\end{proposition}

It has been shown in \cite{Mass14} (see also \cite{Mass30}) that the splittings as $C^\infty(\varM)$-modules of the short exact sequence \eqref{eq-sesderivationsendomorphismalgebra} are in one-to-one correspondence with connections on the vector bundle $\varE$. To any connection $\nabla^\varE$ on $\varE$, it can be associated a unique noncommutative $1$-form $\alpha \in \Omega^1_\der(\algA)$ such that 
\begin{align}
\label{eq-noncommutativeoneformforconnection}
\kX &= \nabla_X - \ad_{\alpha(\kX)}
&
\alpha(\ad_\gamma) &= -\gamma
\end{align}
where $\nabla$ is the connection on the vector bundle $\End(\varE)$ induced by $\nabla^\varE$. This construction leads to the fact that any connection on $\varE$ defines a noncommutative connection on the right $\algA$-module $\modM = \algA$, the one determined by this noncommutative $1$-form $\alpha$ by Proposition~\ref{prop-noncommutativeconnectionsonmodM=algA}. 

In \cite{Mass14}, it is proved that the space of noncommutative connections on the right $\algA$-module $\modM = \algA$ contains the space of connections on $\varE$ and that this inclusion is compatible with the notions of curvature and gauge transformations.

This result must be compared with similar properties carried by connections on transitive Lie algebroids. The relation \eqref{eq-noncommutativeoneformforconnection} is the same as the expression \eqref{eq-decompositionwithconnection} which defines a similar $1$-form on a transitive Lie algebroid. In fact, they both rewrite the splitting given by $\nabla$ of the corresponding short exact sequence in terms of algebraic objects. With the identifications of the spaces of forms, it is clear that the obtained $1$-forms in these two situations are the same: $\alpha \in \Omega^1_\lie(\algA, \algA_0) \subset \Omega^1_\der(\algA)$.

\medskip
Locally, any section $s$ of the vector bundle $\varE$ is trivialized as $s_i : U_i \rightarrow \gC^n$. The gluing relations between these trivializations are given by $s_i = g_{ij} s_j$. For any $\kX \in \der(\algA)$ written locally as $\kX_i = X \oplus \gamma_i$, with $\tr(\gamma_i) = 0$, and for any $s \in \Gamma(\varE)$, a straightforward computation shows that
\begin{equation*}
X \cdotaction s_i + \gamma_i s_i = g_{ij}( X \cdotaction s_j + \gamma_j s_j )
\end{equation*}
which means that there is a well-defined representation $\phi : \der(\algA) \rightarrow \kD(\varE)$. For this representation, $\phi_{\algA_0} : \algA_0 \rightarrow \algA(\varE) = \algA$ is just the inclusion. Let us denote by $(\Omega^\grast_\lie(\algA), \hd_\varE)$ the graded differential algebra $(\Omega^\grast(\der(\algA), \algA(\varE)), \hd_\varE)$. Then a direct comparison between the spaces and the differentials shows the following:

\begin{proposition}
\label{prop-identificationdifferentialcalculirepresentation}
$(\Omega^\grast_\lie(\algA), \hd_\varE)$ and $(\Omega^\grast_\der(\algA), \hdNC)$ are isomorphic as graded differential algebras.
\end{proposition}

As a direct application of this result, one has the following:

\begin{theorem}
\label{thm-derAconnectionsonEandnoncommutativeconnections}
The following three spaces are isomorphic:
\begin{enumerate}
\item The space of $\der(\algA)$-connections on $\varE$ (Lie algebroid of derivations).
\item The space of $\Gamma_G(T\varP)$-connections on $\varE$ (Atiyah Lie algebroid of the $SL(n, \gC)$-principal bundle).
\item The space of noncommutative connections on the right $\algA$-module $\modM = \algA$ (noncommutative differential structure).
\end{enumerate}
These isomorphisms are compatible with curvatures and gauge transformations.
\end{theorem}

\begin{proof}
The isomorphism between 1 and 2 is a trivial consequence of Proposition~\ref{prop-identificationLiealgebroidsDerivationsAtiyah} and the fact that the representations of these two Lie algebroids on $\varE$ are the same as it can be easily checked.

Moreover, a $\der(\algA)$-connection on $\varE$ is completely determined by its associated $1$-form $\omega^\varE \in \Omega^1_\lie(\algA)$ (Proposition~\ref{prop-connectiononeformcurvaturetwoforms}). As a $1$-form in $\Omega^1_\der(\algA)$ (Proposition~\ref{prop-identificationdifferentialcalculirepresentation}), $\omega^\varE$ determines an unique noncommutative connection (Proposition~\ref{prop-noncommutativeconnectionsonmodM=algA}).
As already noticed, the gauge group $\Aut(\varE)$ is the space of invertible elements in $\algA(\varE) = \algA$, so that it coincides with the gauge group of the $\algA$-module $\modM = \algA$ (Proposition~\ref{prop-noncommutativeconnectionsonmodM=algA}). 

The identification between connections is compatible with the various expressions for curvatures and gauge transformations (Propositions~\ref{prop-connectiononeformcurvaturetwoforms}, \ref{prop-gaugegroupactiononlieAconnectionsonvarE} and \ref{prop-noncommutativeconnectionsonmodM=algA}).
\end{proof}

Let us say that a noncommutative connection on the right $\algA$-module $\modM = \algA$ is traceless if its noncommutative $1$-form is traceless. Notice that ordinary connections on $\varE$ give rise to traceless noncommutative connections. Then one has:

\begin{theorem}
\label{thm-GeneralizedconnectionsandtracelessNCconnections}
The following three spaces are isomorphic:
\begin{enumerate}
\item The space of generalized connection $1$-forms on $\der(\algA)$.
\item The space of generalized connection $1$-forms on $\Gamma_G(T\varP)$.
\item The space of traceless noncommutative connections on the right $\algA$-module $\modM = \algA$
\end{enumerate}
The isomorphisms are compatible with curvatures and gauge transformations.
\end{theorem}

\begin{proof}
The proof relies on the same method used in Theorem~\ref{thm-derAconnectionsonEandnoncommutativeconnections}, except that now we identify the space $\Omega^\grast_\lie(\algA, \algA_0)$ as traceless noncommutative forms using Proposition~\ref{prop-relationsbetweendifferentialcalculi}.
\end{proof}

\smallskip
Proposition~\ref{prop-identificationdifferentialcalculusAtiyah} shows that $(\Omega^\grast_\lie(\varP, \kg), \hd)$ can be identified with the basic sub-complex in a trivial differential calculus for a convenient Cartan operation. It has been shown in \cite{Mass15} that the graded differential algebra $(\Omega^\grast_\der(\algA), \hdNC)$ is also characterized by such an inclusion.

In order to see that, one has to introduce a bigger algebra $\algB = C^\infty(\varP) \otimes M_n$, which is the algebra of matrix valued functions on $\varP$. Then the algebra $\algA$ identifies with
\begin{equation*}
\algA \simeq \{ a : \varP \rightarrow M_n \, / \, a(p\cdotaction g) = g^{-1} a(p) g \text{ for all } g \in SL(n, \gC) \} \subset \algB
\end{equation*}
The Lie algebra of derivations of $\algB$ has the following trivial structure: 
\begin{equation*}
\der(\algB) = \Gamma(T\varP) \oplus (C^\infty(\varP) \otimes \ksl_n)
\end{equation*}
where $\kX = \sfX \oplus \gamma$ acts on $b \in \algB$ as $\kX \cdotaction b = \sfX \cdotaction b + [\gamma, b]$. Its derivation-based differential calculus is then the tensor product $\Omega^\grast_\der(\algB) = \Omega^\grast(\varP) \otimes \Omega^\grast_\der(M_n)$ with the differential $\hd_\algB = \dd + \dd'$. The differential calculus $(\Omega^\grast(\varP), \dd)$ is the de~Rham differential calculus on the manifold $\varP$, and $(\Omega^\grast_\der(M_n), \dd')$ is the noncommutative derivation-based differential calculus on the matrix algebra $M_n$. This differential calculus is described in Proposition~\ref{prop-differentialcalculusmatrixalgebra}.

As in \ref{ex-atiyahliealgebroid}, let us define $\kg_\equ = \{ \xi^v + \ad_\xi \ / \ \xi \in \ksl_n \}$. 
\begin{proposition}[\cite{Mass15}]
$(\Omega^\grast_\der(\algA), \hdNC)$ is the graded differential algebra of basic elements for the Cartan operation of $\kg_\equ$ on $(\Omega^\grast_\der(\algB), \hd_\algB)$.
\end{proposition}

Now, with $\evE = \gC^n$, one has $\caL(\evE) = M_n$, so that it is straightforward to check the following:
\begin{proposition}
The differential calculus $(\Omega^\grast_\tla(\varP,\kg,\evE), \hd_\evE)$ is isomorphic to the differential calculus $(\Omega^\grast_\der(\algB), \hd_\algB)$.
\end{proposition} 

Obviously, Proposition~\ref{prop-identificationdifferentialcalculirepresentation} can be proved again with these two last results.

\smallskip
For this particular transitive Lie algebroid, besides the notion of infinitesimal gauge transformations as defined in Definition~\ref{def-infinitesimalgaugetransformations}, one has the notion of gauge transformations when a $\algA$-module is given. These gauge transformations are defined using the noncommutative framework, and they exactly coincide with the usual notion of gauge transformations on the associated principal fiber bundle $\varP$ for the right $\algA$-module $\algA$ itself. However, as for gauge transformations on Lie algebroids, notice that this definition requires a representation of the algebra implementing the underlying symmetry.

\section{Conclusions}

In this paper we have exhibited relations between ordinary connections on principal fiber bundles, connections and their generalizations on transitive Lie algebroids ($\lieA$-connections, generalized connection $1$-forms) and connections in the framework of the derivation-based noncommutative geometry. The Atiyah transitive Lie algebroid appears as the common corner stone on which all these notions of connections turn out to be explicitly formulated and compared.

In particular, we have shown that the space of generalized connections on the Atiyah transitive Lie algebroid of a $SL(n, \gC)$-principal fiber bundle is isomorphic to the space of noncommutative connections on the derivation-based noncommutative geometry of the algebra of endomorphisms of a $\gC^n$-vector bundle associated to this principal fiber bundle. In this correspondence, gauge transformations are compatible.

\medskip
\noindent\textbf{Acknowledgments :}
We thank the referee for suggestions about improvement on the paper and references.

\renewcommand{\thesection}{\Alph{section}}
\setcounter{section}{0}

\section{Brief review of some noncommutative structures}
\label{briefreviewofsomenoncommutativestructures}

In this appendix, we collect some useful definitions and structures introduced in the context of the derivation-based differential noncommutative geometry.  See \cite{DuVi88,Mass14,Mass15,Mass30} for more details.

\subsection{Derivation-based differential calculus}

Let us consider an associative algebra $\algA$ with unit denoted by $\bbbone$. Denote by $\caZ(\algA)$ the center of $\algA$. The differential calculus we consider in this noncommutative geometry is constructed using the space of derivations of $\algA$.

\begin{definition}
The vector space of derivations of $\algA$ is the space
$\der(\algA) = \{ \kX : \algA \rightarrow \algA \ / \ \kX \text{ linear}, \kX\cdotaction (ab) = (\kX\cdotaction a) b + a (\kX\cdotaction b), \forall a,b\in \algA\}$
\end{definition} 

\begin{proposition}
$\der(\algA)$ is a Lie algebra for the bracket defined by $[\kX, \kY ]\cdotaction a = \kX \cdotaction \kY \cdotaction a - \kY \cdotaction \kX \cdotaction a$ for any $\kX,\kY \in \der(\algA)$ and it is a $\caZ(\algA)$-module for the product defined by $(f\kX )\cdotaction a = f(\kX \cdotaction a)$ for any $a \in \algA$, $f \in \caZ(\algA)$ and $\kX \in \der(\algA)$).

The subspace $\Int(\algA) = \{ \ad_a : b \mapsto [a,b]\ / \ a \in \algA\} \subset \der(\algA)$, called the vector space of inner derivations, is a Lie ideal and a $\caZ(\algA)$-submodule.

With $\Out(\algA)=\der(\algA)/\Int(\algA)$, there is a short exact sequence of Lie algebras and $\caZ(\algA)$-modules
\begin{equation}
\label{eq-secderivations}
\xymatrix@1@C=15pt{{\algzero} \ar[r] & {\Int(\algA)} \ar[r] & {\der(\algA)} \ar[r] & {\Out(\algA)} \ar[r] & {\algzero}}
\end{equation}
\end{proposition} 

The two structures exhibited before, Lie algebra and $\caZ(\algA)$-module, are used in the following definition of the derivation-based differential calculus.

\begin{definition}
\label{def-thegradeddifferentialalgebraNCunderline}
Let $\underline{\Omega}^n_\der(\algA)$ be the set of $\caZ(\algA)$-multilinear antisymmetric maps from $\der(\algA)^n$ to $\algA$, with $\underline{\Omega}^0_\der(\algA) = \algA$, and let $\underline{\Omega}^\grast_\der(\algA) =\textstyle \bigoplus_{n \geq 0} \underline{\Omega}^n_\der(\algA)$. We introduce on $\underline{\Omega}^\grast_\der(\algA)$ a structure of $\gN$-graded differential algebra using the product
\begin{multline*}
(\omega\eta)(\kX_1, \dots, \kX_{p+q}) =\\
 \frac{1}{p!q!} \sum_{\sigma\in \kS_{p+q}} (-1)^{\sign(\sigma)} \omega(\kX_{\sigma(1)}, \dots, \kX_{\sigma(p)}) \eta(\kX_{\sigma(p+1)}, \dots, \kX_{\sigma(p+q)})
\end{multline*}
and using the differential $\dd$ (of degree $1$) defined by the Koszul formula
\begin{multline}
\label{eq-differentialNC}
\dd\omega(\kX_1, \dots , \kX_{n+1}) = \sum_{i=1}^{n+1} (-1)^{i+1} \kX_i \cdotaction \omega( \kX_1, \dots \omi{i} \dots, \kX_{n+1}) \\
 + \sum_{1\leq i < j \leq n+1} (-1)^{i+j} \omega( [\kX_i, \kX_j], \dots \omi{i} \dots \omi{j} \dots , \kX_{n+1}) 
\end{multline}
\end{definition} 


This definition is motivated by the following important example:
\begin{example}
Let $\varM$ be a smooth manifold and let $\algA = C^\infty(\varM)$. The center of this algebra is $\algA$ itself: $\caZ(\algA) = C^\infty(\varM)$. The Lie algebra of derivations is exactly the Lie algebra of smooth vector fields on $\varM$: $\der(\algA) = \Gamma(T\varM)$. In that case, there are no inner derivations, $\Int(\algA) = \algzero$, so that $\Out(\algA) = \Gamma(T\varM)$. The graded differential algebra coincides with the graded differential algebra of de~Rham forms on $\varM$: $\underline{\Omega}^\grast_\der(\algA) = \Omega^\grast(\varM)$.
\end{example}

\subsection{Noncommutative connections}
\label{noncommutativeconnections}

Connections on noncommutative algebras are defined using modules, which play the role of representations of the algebra. Let $\modM$ be a right $\algA$-module.

\begin{definition}
\label{def-noncommutativeconnectionsandcurvature}
A noncommutative connection on the right $\algA$-module $\modM$ for the differential calculus based on derivations is a linear map $\widehat{\nabla}_\kX : \modM \rightarrow \modM$, defined for any $\kX \in \der(\algA)$, such that $\forall \kX,\kY \in \der(\algA)$, $\forall a \in \algA$, $\forall m \in \modM$, $\forall f \in \caZ(\algA)$:
\begin{align*}
\widehat{\nabla}_\kX (m a) &= m(\kX \cdotaction a) + (\widehat{\nabla}_\kX m) a,
&
\widehat{\nabla}_{f\kX} m &= f \widehat{\nabla}_\kX m,
&
\widehat{\nabla}_{\kX + \kY} m &= \widehat{\nabla}_\kX m + \widehat{\nabla}_\kY m
\end{align*}
The curvature of $\widehat{\nabla}$ is the linear map $\hR(\kX, \kY) : \modM \rightarrow \modM$ defined for any $\kX, \kY \in \der(\algA)$ by $\hR(\kX, \kY) m = [ \widehat{\nabla}_\kX, \widehat{\nabla}_\kY ] m - \widehat{\nabla}_{[\kX, \kY]}m$.
\end{definition} 

Notice that $\hR(\kX, \kY) : \modM \rightarrow \modM$ is a right $\algA$-module morphism.

\begin{definition}
The gauge group of $\modM$ is the group $\Aut(\modM)$ of the automorphisms of $\modM$ as a right $\algA$-module.
\end{definition} 

\begin{proposition}
For any $\Phi \in \Aut(\modM)$ and any noncommutative connection $\widehat{\nabla}$, the map $\widehat{\nabla}^\Phi_\kX = \Phi^{-1}\circ \widehat{\nabla}_\kX \circ \Phi : \modM \rightarrow \modM$ is a noncommutative connection. This defines the action of the gauge group of $\modM$ on the space of noncommutative connections on $\modM$.
\end{proposition}

As a special case we consider the right $\algA$-module $\modM=\algA$. 

\begin{proposition}
\label{prop-noncommutativeconnectionsonmodM=algA}
A noncommutative connection $\widehat{\nabla}$ on $\modM=\algA$ is completely determined by $\omega \in \underline{\Omega}^1_\der(\algA)$ defined by $\omega(\kX) = \widehat{\nabla}_\kX \bbbone$. For any $a \in \modM=\algA$, one has $\widehat{\nabla}_\kX a = \kX \cdotaction a + \omega(\kX) a$.
The curvature of $\widehat{\nabla}$ is the multiplication on the left on $\modM=\algA$ by the noncommutative $2$-form $\Omega(\kX, \kY) = \dd\omega (\kX, \kY) + [ \omega(\kX), \omega(\kY) ]$.

The gauge group is identified with the invertible elements $g \in \algA$ by $\Phi_g(a) = ga$ and the gauge transformations on $\widehat{\nabla}$ take the explicit form $\omega \mapsto \omega^g = g^{-1} \omega g + g^{-1} \dd g$ and $\Omega \mapsto \Omega^g = g^{-1} \Omega g$.

$\widehat{\nabla}^0_\kX$, defined by $a \mapsto \kX \cdotaction a$, is a noncommutative connection on $\algA$.
\end{proposition}

\subsection{The matrix algebra}
\label{thematrixalgebra}

Let us describe the derivation-based noncommutative geometries of the finite dimensional algebra of complex $n\times n$ matrices $\algA = M_n(\gC) = M_n$. We refer to \cite{DuVi88,DuViKernMado90a,Mass11} for further properties.
\begin{proposition}
\label{prop-differentialcalculusmatrixalgebra}
One has:
\begin{itemize}
\item $\caZ(M_n) = \gC$.

\item $\der(M_n) = \Int(M_n) \simeq \ksl_n =\ksl(n,\gC)$ (traceless matrices). The explicit isomorphism associates to any $\gamma \in \ksl_n(\gC)$ the derivation $\ad_\gamma : a \mapsto [\gamma, a]$. 

\item $\underline{\Omega}^\grast_\der(M_n) \simeq M_n \otimes \exter^\grast \ksl_n^\ast$, with the differential $\dd'$ coming from the differential of the differential complex of the Lie algebra $\ksl_n$ represented on $M_n$ by the adjoint representation (commutator).

\item There exits a canonical noncommutative $1$-form $i\theta \in \underline{\Omega}^1_\der(M_n)$ such that  for any $\gamma \in M_n(\gC)$, $i\theta(\ad_{\gamma}) = \textstyle\gamma - \frac{1}{n} \tr (\gamma)\bbbone$.
This noncommutative $1$-form $i\theta$ makes the explicit isomorphism $\Int(M_n(\gC)) \xrightarrow{\simeq} \ksl_n$.

\item $i\theta$ satisfies the relation $\dd' (i\theta) - (i\theta)^2 = 0$. This makes $i\theta$ looks very much like the Maurer-Cartan form in the geometry of Lie groups (here $SL_n(\gC)$).

\item For any $a \in M_n$, one has $\dd' a = [i\theta, a] \in \underline{\Omega}^1_\der(M_n)$. This relation is no longer true in higher degrees.
\end{itemize}

\end{proposition}

\section*{References}

\bibliography{biblio-articles-perso,biblio-livre,biblio-articles}

\end{document}